\documentclass[12pt,paper]{amsart}
\usepackage{amsmath}
\usepackage{mathrsfs}
\usepackage{amsfonts}
\usepackage{amssymb}
\usepackage[all,cmtip]{xy}           %xypic macro for latex2.09
\usepackage{bbm}
\usepackage{bbding}
\usepackage{txfonts}
\usepackage{changes}
\usepackage[shortlabels]{enumitem}
\usepackage{ifpdf}
\ifpdf
\usepackage[colorlinks,final,backref=page,hyperindex]{hyperref}
\else
\usepackage[colorlinks,final,backref=page,hyperindex,hypertex]{hyperref}
\fi
\usepackage{amscd}
\usepackage{tikz-cd}
\usepackage[active]{srcltx}
\usepackage{tikz}
\usepackage{ulem}
\usetikzlibrary{calc}
\usetikzlibrary{arrows,shapes,chains}

\makeatletter

\newtheorem{thm}{Theorem}[section]
\newtheorem{prop}[thm]{Proposition}

\newtheorem{cor}[thm]{Corollary}
\newtheorem{prop-def}{Proposition-Definition}[section]

\theoremstyle{definition}
\newtheorem{defn}[thm]{Definition}

\newtheorem{remark}[thm]{Remark}

\newcommand{\nc}{\newcommand}

\nc{\delete}[1]{{}}
\nc{\mmargin}[1]{}
\nc{\Alg}{\mathrm{Alg}}
\nc{\NjO}{\mathrm{\mathrm{RB\Omega O}_\lambda}}
\nc{\RBF}{\mathrm{RBAF}}
\nc{\rmH}{\mathrm{H}}
\nc{\DT}{\mathrm{DT}}
\nc{\C}{\mathrm{C}}

\nc{\RBO}{{\mathrm{RBO}_\lambda}}
\nc{\Sh}{\mathrm{Sh}}
\nc{\RBA}{{\mathrm{RBA}_\lambda}}
\nc{\sgn}{\mathrm{sgn}}

\delete{
\nc{\mlabel}[1]{\label{#1}}  % Use this to suppress names
\nc{\mcite}[1]{\cite{#1}}  % Use this to suppress names
\nc{\mref}[1]{\ref{#1}}  % Use this to suppress names
\nc{\bibitem}[1]{\bibitem{#1}} % Use this to show number
}

	\nc{\mlabel}[1]{\label{#1}  % Use the next two lines to show names
		{\hfill \hspace{1cm}{\bf{{\ }\hfill(#1)}}}}
	\nc{\mcite}[1]{\cite{#1}{{\bf{{\ }(#1)}}}}  % Use this lines to show names
	\nc{\mref}[1]{\ref{#1}{{\bf{{\ }(#1)}}}}  % Use this lines to show names
	\nc{\mbibitem}[1]{\bibitem[\bf #1]{#1}} % Use this to show name
	
 \font\cyrs=wncyr7

\nc{\vep}{\varepsilon}
\nc{\bin}[2]{ (_{\stackrel{\scs{#1}}{\scs{#2}}})}  %binomial coeff
\nc{\binc}[2]{(\!\! \begin{array}{c} \scs{#1}\\
		\scs{#2} \end{array}\!\!)}  %binomial coeff
\nc{\bincc}[2]{  ( {\scs{#1} \atop
		\vspace{-1cm}\scs{#2}} )}  %binomial coeff
\nc{\oline}[1]{\overline{#1}}
\nc{\mapm}[1]{\lfloor\!|{#1}|\!\rfloor}
\nc{\bs}{\bar{S}}
\nc{\la}{\longrightarrow}
\nc{\ot}{\otimes}
\nc{\rar}{\rightarrow}
\nc{\lon }{\,\rightarrow\,}
\nc{\dar}{\downarrow}
\nc{\dap}[1]{\downarrow \rlap{$\scriptstyle{#1}$}}
\nc{\defeq}{\stackrel{\rm def}{=}}
\nc{\dis}[1]{\displaystyle{#1}}
\nc{\dotcup}{\ \displaystyle{\bigcup^\bullet}\ }
\nc{\hcm}{\ \hat{,}\ }
\nc{\hts}{\hat{\otimes}}
\nc{\hcirc}{\hat{\circ}}
\nc{\lleft}{[}
\nc{\lright}{]}
\nc{\curlyl}{\left \{ \begin{array}{c} {} \\ {} \end{array}
	\right .  \!\!\!\!\!\!\!}
\nc{\curlyr}{ \!\!\!\!\!\!\!
	\left . \begin{array}{c} {} \\ {} \end{array}
	\right \} }
\nc{\longmid}{\left | \begin{array}{c} {} \\ {} \end{array}
	\right . \!\!\!\!\!\!\!}
\nc{\ora}[1]{\stackrel{#1}{\rar}}
\nc{\ola}[1]{\stackrel{#1}{\la}}%${\Bbb Z}$
\nc{\scs}[1]{\scriptstyle{#1}} \nc{\mrm}[1]{{\rm #1}}
\nc{\dirlim}{\displaystyle{\lim_{\longrightarrow}}\,}
\nc{\invlim}{\displaystyle{\lim_{\longleftarrow}}\,}
\nc{\dislim}[1]{\displaystyle{\lim_{#1}}} \nc{\colim}{\mrm{colim}}
\nc{\mvp}{\vspace{0.3cm}} \nc{\tk}{^{(k)}} \nc{\tp}{^\prime}
\nc{\ttp}{^{\prime\prime}} \nc{\svp}{\vspace{2cm}}
\nc{\vp}{\vspace{8cm}}
\nc{\modg}[1]{\!<\!\!{#1}\!\!>}
\nc{\intg}[1]{F_C(#1)}
\nc{\lmodg}{\!<\!\!}
\nc{\rmodg}{\!\!>\!}
\nc{\cpi}{\widehat{\Pi}}
\nc{\ssha}{{\mbox{\cyrs X}}} %sha as product
\nc{\tsha}{{\mbox{\cyrt X}}}
\nc{\shpr}{\diamond}    %Shuffle product
\nc{\labs}{\mid\!}
\nc{\rabs}{\!\mid}

\nc{\ad}{\mrm{ad}}
\nc{\ann}{\mrm{ann}}
\nc{\Aut}{\mrm{Aut}}
\nc{\bim}{\mbox{-}\mathsf{Bimod}}
\nc{\br}{\mrm{bre}}
\nc{\can}{\mrm{can}}
\nc{\Cont}{\mrm{Cont}}
\nc{\rchar}{\mrm{char}}
\nc{\cok}{\mrm{coker}}
\nc{\de}{\mrm{dep}}
\nc{\dtf}{{R-{\rm tf}}}
\nc{\dtor}{{R-{\rm tor}}}

\nc{\Div}{{\mrm Div}}
\nc{\Diff}{\mrm{DA}}
\nc{\Diffl}{\mathsf{DA}_\lambda}
\nc{\diffo}{{\mathsf{DO}_\lambda}}
\nc{\alg}{\mathsf{Alg}}
\nc{\End}{\mrm{End}}
\nc{\Ext}{\mrm{Ext}}
\nc{\Fil}{\mrm{Fil}}
\nc{\Fr}{\mrm{Fr}}
\nc{\Frob}{\mrm{Frob}}
\nc{\Gal}{\mrm{Gal}}
\nc{\GL}{\mrm{GL}}
\nc{\Hom}{\mrm{Hom_\Omega}}
\nc{\Hoch}{\mrm{Hoch}}

\nc{\hsr}{\mrm{H}}
\nc{\hpol}{\mrm{HP}}
\nc{\id}{\mrm{Id}}
\nc{\im}{\mrm{im}}
\nc{\Id}{\mrm{Id}}
\nc{\ID}{\mrm{ID}}
\nc{\Irr}{\mrm{Irr}}
\nc{\incl}{\mrm{incl}}
\nc{\length}{\mrm{length}}
\nc{\NLSW}{\mrm{NLSW}}
\nc{\Lie}{\mrm{Lie}}
\nc{\mchar}{\rm char}
\nc{\mpart}{\mrm{part}}
\nc{\ql}{{\QQ_\ell}}
\nc{\qp}{{\QQ_p}}
\nc{\rank}{\mrm{rank}}
\nc{\rcot}{\mrm{cot}}
\nc{\rdef}{\mrm{def}}
\nc{\rdiv}{{\rm div}}
\nc{\rtf}{{\rm tf}}
\nc{\rtor}{{\rm tor}}
\nc{\res}{\mrm{res}}
\nc{\SL}{\mrm{SL}}
\nc{\Spec}{\mrm{Spec}}
\nc{\tor}{\mrm{tor}}
\nc{\Tr}{\mrm{Tr}}
\nc{\tr}{\mrm{tr}}
\nc{\wt}{\mrm{wt}}
\def\ot{\otimes}
\nc\blue{\color{blue}}

\nc{\bfk}{{\bf k}}
\nc{\bfone}{{\bf 1}}
\nc{\bfzero}{{\bf 0}}
\nc{\detail}{\marginpar{\bf More detail}
	\noindent{\bf Need more detail!}
	\svp}
\nc{\gap}{\marginpar{\bf Incomplete}\noindent{\bf Incomplete!!}
	\svp}
\nc{\FMod}{\mathbf{FMod}}
\nc{\Int}{\mathbf{Int}}
\nc{\Mon}{\mathbf{Mon}}
 \nc{\sproof}{\noindent{  \textit{Sketch of Proof:} }}
\nc{\remarks}{\noindent{\bf Remarks: }}
\nc{\Rep}{\mathbf{Rep}}
\nc{\Rings}{\mathbf{Rings}}
\nc{\Sets}{\mathbf{Sets}}
\nc{\ob}{\mathsf{Ob}}
\nc{\cal}{\mathcal}

\nc{\BA}{{\mathbb A}}   \nc{\CC}{{\mathbb C}}
\nc{\DD}{{\mathbb D}}   \nc{\EE}{{\mathbb E}}
\nc{\FF}{{\mathbb F}}   \nc{\GG}{{\mathbb G}}
    \nc{\LL}{{\mathbb L}}
\nc{\NN}{{\mathbb N}}   \nc{\PP}{{\mathbb P}}
\nc{\QQ}{{\mathbb Q}}   \nc{\RR}{{\mathbb R}}
\nc{\TT}{{\mathbb T}}   \nc{\VV}{{\mathbb V}}
\nc{\ZZ}{{\mathbb Z}}   \nc{\TP}{\widetilde{P}}
\nc{\m}{{\mathbbm m}}

\nc{\cala}{{\mathcal A}}    \nc{\calc}{{\mathcal C}}
\nc{\cald}{\mathcal{D}}     \nc{\cale}{{\mathcal E}}
\nc{\calf}{{\mathcal F}}    \nc{\calg}{{\mathcal G}}
\nc{\calh}{{\mathcal H}}    \nc{\cali}{{\mathcal I}}
\nc{\call}{{\mathcal L}}    \nc{\calm}{{\mathcal M}}
\nc{\caln}{{\mathcal N}}    \nc{\calo}{{\mathcal O}}
\nc{\calp}{{\mathcal P}}    \nc{\calr}{{\mathcal R}}
\nc{\cals}{{\mathcal S}}    \nc{\calt}{{\Omega}}
\nc{\calv}{{\mathcal V}}    \nc{\calw}{{\mathcal W}}
\nc{\calx}{{\mathcal X}}

\nc{\fraka}{{\mathfrak a}}
\nc{\frakb}{\mathfrak{b}}
\nc{\frakg}{{\frak g}}
\nc{\frakl}{{\frak l}}
\nc{\fraks}{{\frak s}}
\nc{\frakB}{{\frak B}}
\nc{\frakm}{{\frak m}}
\nc{\frakM}{{\frak M}}
\nc{\frakp}{{\frak p}}
\nc{\frakW}{{\frak W}}
\nc{\frakX}{{\frak X}}
\nc{\frakS}{{\frak S}}
\nc{\frakA}{{\frak A}}
\nc{\frakC}{{\frak C}}
\nc{\frakx}{{\frakx}}
\nc{\red}{\color{red}}
\nc{\RB}{\mathfrak{RB}}
\nc{\lbar}[1]{\overline{#1}}
\nc{\ra}{\rightarrow}
\nc{\hm}{\mrm{Hom}}

\nc{\tred}[1]{\textcolor{red}{#1}} \nc{\tgreen}[1]{\textcolor{green}{#1}}
\nc{\tblue}[1]{\textcolor{blue}{#1}} \nc{\tpurple}[1]{\textcolor{purple}{#1}}

\nc{\yuan}[1]{\tred{\underline{Yuan:}#1 }}
\nc{\kai}[1]{\tblue{\underline{kai:}#1 }}
\nc{\chao}[1]{\tblue{\underline{Chao:}#1 }}
\begin{document}

\title[Deformations and cohomology theory of $\Omega$-family Rota-Baxter algebras]{Deformations and cohomology theory of $\Omega$-family Rota-Baxter algebras of arbitrary weight}

\author{Chao Song}
\address{School of Mathematical Sciences, Shanghai Key Laboratory of PMMP, East China Normal University, Shanghai 200241, China}
\email{52265500011@stu.ecnu.edu.cn}

\author{Kai Wang}
\address{School of Mathematical Sciences, Shanghai Key Laboratory of PMMP, East China Normal University, Shanghai 200241, China}
\email{wangkai@math.ecnu.edu.cn }

\author{Yuanyuan Zhang$^{*}$}\thanks{*Corresponding author}
\address{School of Mathematics and Statistics, Henan University, Henan, Kaifeng 475004, China}
\email{zhangyy17@henu.edu.cn}
%

%\author{Chao Song^a, Kai Wang^a, Yuanyuan Zhang^b,}
%\address{School of Mathematical Sciences, Shanghai Key Laboratory of PMMP,
%  East China Normal University, Shanghai 200241,
%   China}
%\address{a School of Mathematical Sciences, Shanghai Key Laboratory of PMMP,
%  East China Normal University,
% Shanghai 200241,
%   China}

%\email{zhangyy17@henu.edu.cn }
%
%\email{gdzhou@math.ecnu.edu.cn}

\date{\today}

\begin{abstract}  In this paper, we firstly construct an $L_\infty[1]$-algebra via the method of higher derived brackets, whose Maurer-Cartan elements correspond to  relative $\Omega$-family Rota-Baxter algebras structures of weight $\lambda$. For a relative $\Omega$-family Rota-Baxter algebra of weight $\lambda$, the corresponding twisted $ L_{\infty}[1] $-algebra controls its deformations, which leads to the cohomology theory of relative $\Omega$-family Rota-Baxter algebras of weight $\lambda$. 	Moreover, we also obtain the corresponding results for absolute $\Omega$-family Rota-Baxter algebras of weight $\lambda$  from the relative version.
At last, we study formal deformations of  relative (resp. absolute) $\Omega$-family Rota-Baxter algebras of weight $\lambda$, which can be explained by the lower degree cohomology groups.
% and abelian extensions.
\end{abstract}

\subjclass[2020]{
%17B38  %Yang-Baxter equations and Rota-Baxter family operator
16E40   %(co)homology of rings and algebras
16S80   %deformations of rings
%12H05   %differential algebra
%12H10   %difference algebra
%16W25   %derivations
%16S70   %extensions of rings by ideas
17B38 %Yang-Baxter equations and Rota-Baxter opera-tors
}

\keywords{$L_{\infty}$[1]-algebra, cohomology,
%	abelian extension,
	formal deformation, $\Omega$-family Rota-Baxter algebra, $\Omega$-associative algebra}

\maketitle

 \tableofcontents

\allowdisplaybreaks

\section{Introduction}\
%
%\subsection{Background}
{\bf Background:}
Rota-Baxter algebra was introduced by Baxter~\cite{Bax} in 1960 in his study of fluctuation theory in probability. Later Baxter's work was further explored from different angles by Rota~\cite{Rot69a, Rot69b, Rota95}, Cartier~\cite{Car} and Atkinson~\cite{Atk63} among others in the 1960-70s.
In 1980s,  Semenov-Tian-Shansky~\cite{STS} studied classical Yang-Baxter equation (CYBE) systematically and he proved that a Rota-Baxter operator of weight $0$ on a Lie algebra is exactly the operator form of a skew-symmetric solution of CYBE. Later, Kupershmidt~\cite{Kup} introduced the notion of $\mathcal{O}$-operator (so-called relative Rota-Baxter operator later) on Lie algebras and built the connections between solutions of generalized CYBE and $\mathcal{O}$-operators. Relative Rota-Baxter operators on Lie algebras also  provide solutions of the classical Yang-Baxter equation in the semidirect product Lie algebra and give rise to pre-Lie algebras~\cite{Bai}. In the context of associative algebras, Aguiar~\cite{Aguiar20} established the connection between solutions of associative Yang-Baxter equation and Rota-Baxter operators of weight $0$ on associative algebras.  Uchino~\cite{Uchino08} introduced the notion of relative Rota-Baxter operators on associative algebras . Bai, Guo and Ni~\cite{BGN} built the correspondence between solutions of extended Yang-Baxter equation and relative Rota-Baxter operators on associative algebras. Nowadays, Rota-Baxter algebras have numerous connections and applications to many areas of mathematics, including combinatorics~\cite{Rot69a, Rot69b, Rota95, Spitzer}, operad theory~\cite{BBG}, quantum field theory~\cite{CK20,EGK,GPZ}, pre-Lie and pre-Poisson algebras~\cite{Aguiar20,AB08}, Loday's dendriform algebras~\cite{EG08, EM08, Guo12, LR04}, shuffle algebras~\cite{GK20} etc.

\smallskip
The concept of algebras with multiple linear operators (also called $\Omega$-algebra) was first introduced by Kurosch in ~\cite{Kur}. The first example of this situation appeared in 2007 in a paper by  Ebrahimi-Fard, Gracia-Bondia and Patras \cite[Proposition~9.1]{FBP} (see also \cite[Theorem 3.7.2]{DK}) about algebraic aspects of renormalization in Quantum Field Theory, where a ``Rota-Baxter family" appears: this terminology was suggested to the authors by  Guo (see Footnote following Proposition 9.2 therein), who further discussed the underlying structure under the name \textsl{Rota-Baxter Family algebra} in \cite{Guo09}. Various other kinds of family algebraic structures have been recently defined \cite{Foi20, ZG, ZGM, ZGM23, ZM}. Let $\Omega$ be a semigroup and $\lambda\in \bfk$ be given.
A Rota-Baxter family of weight $\lambda$ on an associative algebra $A$ is a collection of linear operators $T = \{T_\omega\}_{\omega\in\Omega}:A\ra A$ such that
\begin{equation*}
T_{\alpha}(a)T_{\beta}(b)=T_{\alpha\beta}\big( T_{\alpha}(a)b  + a T_{\beta}(b) + \lambda ab \big),\, \text{ for }\, a, b \in A\,\text{ and }\, \alpha,\, \beta \in \Omega.
\end{equation*}
Then $(A, \mu, T=\{T_\omega\}_{\omega\in\Omega})$ is called a Rota-Baxter family algebra of weight $\lambda$.
 The concept of Rota-Baxter family algebra is a generalization of Rota-Baxter algebras~\cite{Bax}. Recently, many scholars have begun to pay attention to families algebraic structures, such as Foissy~\cite{Foi18, Foi20, FP}, Gao, Guo, Manchon and Zhang~\cite{GGZ21,ZGG,ZGM, ZG, ZGM23, ZM}, Aguiar~\cite{Agu20}, Das~\cite{Das} and so on.
%and the author noticed that the free $\Omega$-algebra carries lots of combinatorial properties. In recent years, $\Omega$-algebra has been widely studied~\cite{20, BoCh, Guo09, Foi18, Foi20, FP, GGZ20, GGZ21, GZ, ZGM, ZM}.

\smallskip
The formal deformation theory of algebraic structures was first developed for associative algebras in the classical work of Gerstenhaber~\cite{Ger1963,Ger}, which is closely related to the cohomology theory of Hochschild cohomology on associative algebras~\cite{Hoch}. In~\cite{Ger}, Gerstenhaber showed that the deformation is governed by the Hochschild cohomology.
 Recently, Lazarev, Sheng and Tang~\cite{LST} defined the cohomology of relative Rota–Baxter Lie algebras and related it to their infinitesimal deformations.
 Tang, Bai, Guo and Sheng~\cite{RBGS} developed the deformation theory and cohomology theory of $\cal{O}$-operators (also called relative Rota-Baxter operator of weight $0$) on Lie algebras. Das~\cite{Das1} developed the corresponding cohomology theory for Rota-Baxter associative algebra of weight $0$. Wang and Zhou~\cite{WZ1, WZ2} defined the cohomology theory for Rota-Baxter associative algebras of weight $\lambda$, determined the underlying $L_\infty$-algebra and also showed that the corresponding dg operad of homotopy Rota-Baxter associative algebras of weight $\lambda$ is the minimal model of that of Rota-Baxter associative algebras of weight $\lambda$. Das~\cite{Das} defined the cohomology of twisted $\cal{O}$-operator family and NS-family algebras (in particular cohomology of Rota-Baxter family and dendriform family algebras) that govern their deformations. Recently, Das~\cite{Das22} also studied deformations and homotopy theory for Rota-Baxter family algebras of weight $0$.

%\subsection{Outline of the paper}

{\bf Outline of the paper:} In this paper, we first construct an $L_\infty[1]$-structure by higher derived brackets and the cohomology theory of relative $\Omega$-family Rota-Baxter algebras of weight $\lambda$. As an application, we can immediately obtain the corresponding theory for the absolute version of this theory.
This paper is organized as follows. In Section~\ref{sec:relative cohomology},
we first construct an $L_\infty[1]$-algebra by higher derived brackets whose Maurer-Cartan elements correspond to relative $\Omega$-family Rota-Baxter algebras structures of weight $\lambda$, then we give the cohomology theory of relative $\Omega$-family Rota-Baxter algebras of weight $\lambda$ by this characterization. In Section~\ref {sec:abosulte cohomology}, applying the results in Section \ref{sec:relative cohomology}, we obtain the $L_\infty[1]$-structure and the cohomology theory of absolute $\Omega$-family Rota-Baxter algebras of weight $\lambda$.
In Section~\ref{sec:formal deformations}, we study formal deformations of relative (resp. absolute) $\Omega$-family Rota-Baxter algebras of weight $\lambda$ and they can be interpreted by lower degree cohomology groups.
\subsection{Notations and conventions}\

Throughout this paper, let $\mathbf{k}$ be a field of characteristic 0. All vector spaces are defined over $\mathbf{k}$, all tensor products and Hom-spaces are taken over $\mathbf{k}$ and we always assume that $\Omega$ is a semigroup.

A graded vector space is a family of vector spaces $V=\left\{V_n\right\}_{n \in \mathbb{Z}}$ indexed by integers. For any $n\in \mathbb{Z}$, an element  $ v\in V_n$ is called  homogeneous of degree $n$, and it is written as $|v|=n$.

For, $n \in \mathbb{N}^+$, let $\mathbb{S}_{n}$ denote the symmetric group in $n$ variables. For $0 \leqslant i_1, \ldots, i_r \leqslant n$ with $i_1+\dots+i_r=n, \operatorname{Sh}\left(i_1, i_2, \ldots, i_r\right)$ is the set of $\left(i_1, \ldots, i_r\right)$-shuffles, i.e., those permutations $\sigma \in \mathbb{S}_{n}$ such that
$$
\sigma(1)<\sigma(2)<\dots<\sigma\left(i_1\right), \sigma\left(i_1+1\right)<\dots<\sigma\left(i_1+i_2\right), \ldots, \sigma\left(i_{r-1}+1\right)<\dots<\sigma(n).
$$

Let $V$ be a graded vector space. Define the graded symmetric algebra $S(V)$ of $V$ to be $T(V) / I$ where $T(V)$ is the tensor algebra and $I$ is a two-sided ideal of $T(V)$ generated by $x \otimes y-(-1)^{|x| |y|} y \otimes x$ for all homogeneous elements $x, y \in V$. For $x_1 \otimes \dots \otimes x_n \in T(V)$, write $x_1 \odot \dots \odot x_n$ to be the corresponding element in $S(V)$. The degree of $x_1 \odot \dots \odot x_n$ is the sum of the degrees of $x_i$.
%Define the weight of $x_1 \odot \dots \odot x_n$ to be $n$, so $S(V)$ is weight graded whose weight $n$-th component is written as $S(V)^{(n)}, n \geqslant 0$.
For homogeneous elements $x_1, \ldots, x_n \in V$ and $\sigma \in \mathbb{S}_n$, the Koszul sign  $\varepsilon\left(\sigma ; x_1, \ldots, x_n\right)$ is defined by
$$
x_1 \odot x_2 \odot \dots \odot x_n=\varepsilon\left(\sigma ; x_1, \ldots, x_n\right) x_{\sigma(1)} \odot x_{\sigma(2)} \odot \dots \odot x_{\sigma(n)} \in S(V).
$$

The suspension operator $s$ changes the grading of $V$ according to the rule $(s V)^i:=V^{i-1}$. The degree $1$ map $s: V \rightarrow sV$ is defined by sending $v \in V$ to its copy $s v \in s V$. The desuspension operator $s^{-1}$ changes the grading of $V$ according to the rule $(s^{-1} V)^i:=V^{i+1}$. The degree $-1$ map $s^{-1}: V \rightarrow s^{-1} V$ is defined by sending $v \in V$ to its copy $s^{-1} v \in s^{-1} V$.
\subsection{$L_{\infty}[1]$-algebras and higher derived brackets}\

In this subsection, we mainly recall the definitions of $L_{\infty}[1]$-algebras and higher derived brackets from\cite{Stasheff, LS, LM, RBGS}.

\begin{defn}\label{L infty def}%\cite{RBGS}
	An $L_{\infty}[1]$-algebra is a graded vector space $\mathfrak{L}=\bigoplus\limits_{i \in \mathbb{Z}} \mathfrak{L}^i$ endowed with a family of graded linear maps $l_n: \mathfrak{L}^{\otimes n} \rightarrow \mathfrak{L}, n \geqslant 1$ of degree 1 satisfying the following equations: for arbitrary $n \geqslant 1$ and $x_1, \ldots, x_n \in \mathfrak{L}$,
	\begin{itemize}
		\item [(i)] (graded symmetry)
		\begin{equation*}\label{L infty 1}
			l_n\left(x_{\sigma(1)}, \ldots, x_{\sigma(n)}\right)=\varepsilon(\sigma) l_n\left(x_1, \ldots, x_n\right), \forall \sigma \in \mathbb{S}_n ;
		\end{equation*}
		\item [(ii)] (generalised Jacobi identity)
		\begin{equation*}\label{L infty 2}
			\sum_{i=1}^n \sum_{\sigma \in \operatorname{Sh}(i, n-i)} \varepsilon(\sigma) l_{n-i+1}\left(l_i\left(x_{\sigma(1)}, \ldots, x_{\sigma(i)}\right), x_{\sigma(i+1)}, \ldots, x_{\sigma(n)}\right)=0 .
		\end{equation*}
		where $\operatorname{Sh}(i, n-i)$ is the set of $(i, n-i)$ shuffles.
	\end{itemize}
\end{defn}

\begin{defn}\label{L infty MC equation}%\cite{RBGS}
	A Maurer-Cartan element of an $L_{\infty}[1]$-algebra $\left(\mathfrak{L},\left\{l_n\right\}_{n \geqslant 1}\right)$ is an element $\alpha \in \mathfrak{L}^{0} $ satisfying the Maurer-Cartan equation:
	$$
	\sum_{n=1}^{\infty} \frac{1}{n !} l_n\left(\alpha^{\otimes n}\right)=0,
	$$
	whenever this infinite sum exists.
	%	 Denote the set of all Maurer-Cartan elements by $ \mathrm{MC}(\mathfrak{L}) $.
\end{defn}

\begin{prop}[Twisting procedure]\label{twisted L infty}%\cite{RBGS}
	Let $\alpha$ be a Maurer-Cartan element of $L_{\infty}[1]$-algebra $\mathfrak{L}$. The twisted $L_{\infty}[1]$-algebra is given by $l_n^\alpha: \mathfrak{L}^{\otimes n} \rightarrow \mathfrak{L}$ which is defined as follows:
	\[
	l_n^\alpha\big(x_1 \otimes \dots \otimes x_n\big)=\sum_{i=0}^{\infty} \frac{1}{i !} l_{n+i}\big(\alpha^{\otimes i} \otimes x_1 \otimes \dots \otimes x_n\big), \forall x_1, \ldots, x_n \in \mathfrak{L},
	\]
	whenever these infinite sums exist. Moreover, if $ \alpha + \widehat{\alpha} $ is a Maurer-Cartan element of $\left(\mathfrak{L},\left\{l_n\right\}_{n \geqslant 1}\right)$, then $ \widehat{\alpha} $ is a Maurer-Cartan element of $\left(\mathfrak{L},\left\{l_n^\alpha \right\}_{n \geqslant 1}\right)$.
\end{prop}

%\chao{Shall we keep or leave out these references?}
\begin{defn}
	%	\cite{RBGS}
	A $V$-data consists of a quadruple $(\mathfrak{L}, \mathfrak{a}, P, \Delta)$ where
	\begin{itemize}
		\item [(i)]$(\mathfrak{L},[-, -])$ is a graded Lie algebra;
		\item [(ii)]$\mathfrak{a}$ is an abelian graded Lie subalgebra of $(\mathfrak{L},[-, -])$;
		\item [(iii)]$P: \mathfrak{L} \rightarrow \mathfrak{L}$ is a projection, that is $P \circ P=P$, whose image is $\mathfrak{a}$ and kernel is a graded Lie subalgebra of $(\mathfrak{L},[-, -])$;
		\item [(iv)]$\Delta$ is an element in $\operatorname{ker}(P)$ with degree $ |\Delta| = 1 $ such that $[\Delta, \Delta]=0$.
	\end{itemize}
\end{defn}

\begin{thm}\label{thm:derived bracket}
	%	\cite{RBGS}
	Let $(\mathfrak{L}, \mathfrak{a}, P, \Delta)$ be a $V$-data. Then the graded vector space $s^{-1} \mathfrak{L} \oplus \mathfrak{a}$ is an $L_{\infty}[1]$-algebra, in which
	\[
	\begin{aligned}
		l_1\big(s^{-1} x, a\big) & =\big(-s^{-1}[\Delta, x], P(x+[\Delta, a])\big); \\
		l_2\big(s^{-1} x, s^{-1} y\big) & =(-1)^x s^{-1}[x, y]; \\
		l_k\big(s^{-1} x, a_1, \dots, a_{k-1}\big) & =P\Big[\dots \big[[x, a_1], a_2 \big] \dots, a_{k-1}\Big],  k \geqslant 2; \\
		l_k\big(a_1, \dots, a_{k-1}, a_k\big) & =P\Big[\dots \big[[\Delta, a_1], a_2 \big] \dots, a_k\Big], k \geqslant 2 .
	\end{aligned}
	\]
	Here $x, y$ are homogeneous elements of $\mathfrak{L}$ and $a, a_1, \dots, a_k$ are homogeneous elements of $\mathfrak{a}$. All the other $L_{\infty}[1]$-algebra products that are not obtained from the ones written above by permutations of arguments, will vanish.
\end{thm}

\begin{remark}\label{rk:subalgebra}
	%	\cite{RBGS}
	Let $\mathfrak{L}^{\prime}$ be a graded Lie subalgebra of $\mathfrak{L}$ that satisfies $\big[\Delta, \mathfrak{L}^{\prime}\big] \subset \mathfrak{L}^{\prime}$. Then $s^{-1} \mathfrak{L}^{\prime} \oplus \mathfrak{a}$ is an $L_{\infty}[1]$-subalgebra of the above $L_{\infty}[1]$-algebra $\big(s^{-1} \mathfrak{L} \oplus \mathfrak{a},\big\{l_k\big\}_{k=1}^{+\infty}\big)$.
\end{remark}
\subsection{$\Omega$-Gerstenhaber bracket and Hochschild cohomology of $\Omega$-associative algebras}\

In this subsection, we mainly recall some basic concepts from~\cite{Agu20, Ger1963, Das22}.

Let $V$ be any vector space. For any $ n \geqslant 1 $,
% \chao{we only need positive integer} any natural number $n \in \mathbb{N}$,
let $\operatorname{Hom}_{\Omega}\left(V^{\otimes n}, V\right)$ be the set whose elements are given by a family $f=\left\{f_{\alpha_1, \ldots, \alpha_n}: V^{\otimes n} \rightarrow V\right\}_{\alpha_1, \ldots, \alpha_n \in \Omega}$ of linear maps labelled by the elements of $\Omega^{\times n}$. For any $n, m_{1}, \dots, m_{n} \geqslant 1$, one can define a linear map
\begin{align*}
	\operatorname{Hom}_{\Omega}(V^{\otimes n}, V) \ \otimes \operatorname{Hom}_{\Omega}(V^{\otimes m_{1}}, V)  \ot \dots \otimes \operatorname{Hom}_{\Omega}(V^{\otimes m_{n}}, V) \rightarrow \operatorname{Hom}_{\Omega}(V^{\otimes m_{1} + \dots + m_{n}}, V)
\end{align*}
\[ f \ot g_{1} \ot \dots \ot g_{n} \mapsto f \circ (g_{1}, \dots, g_{n}) \]
given by
\begin{eqnarray*}
	f\circ (g_1\ot\ldots\ot g_n)
	&=&\{(f\circ (g_1\ot\ldots\ot g_n))
	_{\alpha_{i_{1}},  \dots,   \alpha _{i_{m_{1}}},
		\dots , \alpha_{j_{1}},  \dots ,  \alpha _{j_{m_{n}}} }
	\}
	_{\alpha_{i_{1}},  \dots,   \alpha _{i_{m_{1}}},
		\dots , \alpha_{j_{1}},  \dots ,  \alpha _{j_{m_{n}}} \in \Omega}\\
	&:=&\{f_{\alpha_{i_{1}}  \dots   \alpha _{i_{m_{1}}},
		\dots , \alpha_{j_{1}}  \dots   \alpha _{j_{m_{n}}}}
	\circ
	({g_1}_{\alpha_{i_{1}}, \dots , \alpha _{i_{m_{1}}}}
	\ot\ldots\ot
	{g_n}_{\alpha_{j_{1}}, \dots , \alpha _{j_{m_{n}}}})\}
	_{\alpha_{i_{1}},  \dots,   \alpha _{i_{m_{1}}},
		\dots , \alpha_{j_{1}},  \dots ,  \alpha_{j_{m_{n}}} \in \Omega}.
\end{eqnarray*}
More precisely,
\begin{align*}
	&(f\circ (g_1\ot\ldots\ot g_n))
	_{ \alpha_{i_{1}}, \dots, \alpha_{i_{m_{1}}},
		\dots, \alpha_{j_{1}}, \dots, \alpha_{j_{m_{n}}} }
	(v_{1}, \dots, v_{m_{1}+\dots+m_{n}})\\
	=& f
	_{ \alpha_{i_{1}} \dots \alpha_{i_{m_{1}}},
		\dots, \alpha_{j_{1}} \dots \alpha_{j_{m_{n}}} }
	\big(
	{g_{1}}_{\alpha_{i_{1}}, \dots, \alpha_{i_{m_{1}}}}(v_{1},\dots,v_{m_{1}}), \dots, {g_{n}}_{\alpha_{j_{1}}, \dots, \alpha_{j_{m_{n}}}}(v_{m_{1}+\dots+m_{n-1}+1},\dots,v_{m_{1}+\dots+m_{n}})
	\big)
\end{align*}
for $\alpha_{i_{1}}, \dots, \alpha_{i_{m_{1}}},
\dots, \alpha_{j_{1}}, \dots, \alpha_{j_{m_{n}}} \in \Omega$ and $v_1, \ldots, v_{m_{1}+\dots+m_{n}} \in V$.

For $ n , m \geqslant 1 $, $ 1 \leqslant i \leqslant n $, one can also define a linear map
\[  \operatorname{Hom}_{\Omega}(V^{\otimes n}, V) \ot  \operatorname{Hom}_{\Omega}(V^{\otimes m}, V)  \rightarrow  \operatorname{Hom}_{\Omega}(V^{\otimes n+m-1}, V)\]
\begin{eqnarray*}
	f \ot g \mapsto f\circ_{i} g
	:= f \circ (\id \ot \dots \ot \id \ot \underset{i}{g} \ot \dots \ot \id).
\end{eqnarray*}
In other words,
\begin{eqnarray*}
	f\circ_{i} g
	&=&\{(f\circ_{i} g)
	_{\alpha_{1},  \dots,   \alpha _{n+m-1}}
	\}
	_{\alpha_{1},  \dots,   \alpha _{n+m-1}  \in \Omega}\\
	&:=&
	\{ f_{\alpha_{1},  \dots, \alpha_{i-1}, \alpha_{i} \dots \alpha_{i+m-1}, \alpha_{i+m} \dots,  \alpha _{n+m-1}}
	\circ_{i} g_{\alpha_{i}, \dots, \alpha_{i+m-1}}
	\}
	_{\alpha_{1},  \dots,   \alpha _{n+m-1} \in \Omega},
\end{eqnarray*}
specifically,
\begin{equation*}\label{eq:circ}
	\begin{aligned}
		& \left(f \circ_i g\right)_{\alpha_1, \ldots, \alpha_{n+m-1}}\left(v_1, \ldots, v_{n+m-1}\right) \\
		=&f_{\alpha_1,\, \ldots,\alpha_{i-1}, \alpha_i \dots \alpha_{i+m-1}, \ldots, \alpha_{n+m-1}}
		\left(v_1, \ldots, v_{i-1}, g_{\alpha_i, \ldots, \alpha_{i+m-1}}\left(v_i, \ldots, v_{i+m-1}\right), v_{i+m}, \ldots, v_{n+m-1}\right),
	\end{aligned}
\end{equation*}
for $\alpha_1, \ldots, \alpha_{n+m-1} \in \Omega$ and $v_1, \ldots, v_{n+m-1} \in V$.

%\chao{ $\oplus_{n \geqslant 0} \operatorname{Hom}_{\Omega}\left(V^{\otimes n+1}, V\right)$ ?}

As a consequence, the graded vector space $\bigoplus\limits_{n \geqslant 1} \operatorname{Hom}_{\Omega}\left(V^{\otimes n}, V\right)$ carries graded Lie bracket of degree $-1$ (called the {\bf $\Omega$-Gerstenhaber bracket}) given by
\[
[f, g]_{\Omega}:=\sum_{i=1}^m(-1)^{(i-1)(n-1)} f \circ_i g-(-1)^{(m-1)(n-1)} \sum_{i=1}^n(-1)^{(i-1)(m-1)} g \circ_i f,
\]
for $f \in \operatorname{Hom}_{\Omega}\left(V^{\otimes m}, V\right)$ and $g \in \operatorname{Hom}_{\Omega}\left(V^{\otimes n}, V\right)$.

\begin{defn}  % \cite{Agu20}
	%Let $\Omega$ be a semigroup.
	An associative algebra relative to the semigroup $\Omega$ %({\bf $\Omega$-associative algebra})
	is a vector space $A$ together with a family of bilinear operations
	%	$ \mu = (\mu_{\alpha,\, \beta})_{\alpha,\,\beta\in\Omega} $
	$ \mu = \{\mu_{\alpha,\, \beta}\}_{\alpha,\,\beta\in\Omega} $
	such that
	\[ \mu_{\alpha,\, \beta} : A \otimes  A \rightarrow  A,\,a\ot b\mapsto a\cdot_{\alpha,\,\beta} b\]
	satisfying
	\begin{align*}
		(a \cdot_{\alpha,\, \beta} b) \cdot_{\alpha \beta, \gamma} c = a \cdot_{\alpha,\, \beta \gamma} (b \cdot_{\beta, \gamma} c)
	\end{align*}
	for $ a, b, c \in A $ and $ \alpha, \beta, \gamma \in \Omega $.
	In this case, we call $ (A , \mu = \{\mu_{\alpha,\, \beta}\}_{\alpha,\,\beta\in\Omega}) $ an \textbf{$ \Omega $-associative algebra}.
\end{defn}

\begin{defn}\label{defn:bimodule}
	Let $(A , \mu)$
%	$(A , \mu = \{\mu_{\alpha,\, \beta}\}_{\alpha,\,\beta\in\Omega})$
	be an $\Omega$-associative algebra. A {\bf bimodule} over it consists of a vector space $M$ together with two families of linear maps $ l=\{l_{\alpha,\,\beta}\}_{\alpha,\,\beta \in \Omega} $ and $ r=\{r_{\alpha,\,\beta}\}_{\alpha,\,\beta \in \Omega} $ with
%	 indexded by $\Omega\times \Omega$,
	\begin{align*}
		l_{\alpha,\,\beta}: A \otimes M \rightarrow M ,& \quad (a, m) \mapsto a \cdot_{\alpha,\,\beta} m \\
		r_{\alpha,\,\beta}: M \otimes A \rightarrow M,  &  \quad (m, a) \mapsto m \cdot_{\alpha,\, \beta} a
	\end{align*}
	satisfying
	\begin{align*}
		(a \cdot_{\alpha,\, \beta} b) \cdot_{\alpha \beta,\, \gamma} m =& \, a \cdot_{\alpha,\, \beta \gamma} (b \cdot_{\beta, \gamma} m), \quad \\
		(a \cdot_{\alpha,\, \beta} m) \cdot_{\alpha \beta,\, \gamma} b =& \, a \cdot_{\alpha,\, \beta \gamma} (m \cdot_{\beta, \gamma} b),\\
		(m \cdot_{\alpha,\, \beta} a) \cdot_{\alpha \beta,\, \gamma} b =& \, m \cdot_{\alpha,\, \beta \gamma} (a \cdot_{\beta,\, \gamma} b).
	\end{align*}
for $a, b \in A, m \in M \text{ and } \alpha, \beta, \gamma \in \Omega$.
\end{defn}

The following result is well known.
\begin{prop}%\cite{Ger1963}
%	$(A, \mu=(\mu_{\alpha,\,\beta})_{\alpha,\,\beta\in\Omega})$
Let $A$ be a vector space and $\mu: A^{\ot 2}\rightarrow A$ be a linear map. Then $(A, \mu)$ is an $\Omega$-associative algebra, if and only if, $\mu \in \operatorname{Hom}_\Omega\big(A^{\otimes 2}, A\big)$ is a Maurer-Cartan element in the graded Lie algebra $\big(\bigoplus\limits_{n \geqslant 0} \operatorname{Hom}_\Omega\big(A^{\otimes n+1}, A\big),[-,-]_{\Omega}\big)$, i,e. $ [\mu,\mu]_{\Omega} = 0$.
\end{prop}

\begin{defn}
	%Let $\Omega$ be a semigroup. \chao{ $ A $, $ \Omega $-associative  algebra, M bimodule}
	Let
	%	$(A, \mu=(\mu_{\alpha,\,\beta})_{\alpha,\,\beta\in\Omega})$
	$(A, \mu)$ be an $\Omega$-associative algebra, $ M $ be an $ A $-bimodule.
	%\begin{enumerate}
	The {\bf Hochschild cochain complex of $A$ with coefficients in $M$}, denoted by $\mathcal{C}_{\Alg}^{\bullet}(A,\, M)$, is defined to be the graded space $\bigoplus\limits_{n\geqslant 1}\Hom(A^{\ot n},M)$ endowed with the following coboundary operator: \[\delta^{M}_{\Alg}: \Hom(A^{\otimes n}, M ) \ra \operatorname{Hom}_\Omega(A^{\otimes n+1}, M)\]
	% The coboundary operator \[\delta^{M}_{\Alg}= \delta^{n, M}_{\Alg}: \mathcal{C}_{\Alg}^{n}(A,\, M)\ra \mathcal{C}_{\Alg}^{n+1}(A,\, M)\]
	given as
	\begin{align*}
		%\big(   \delta_{\Alg} (m) \big)_\alpha (a) =~& - a \cdot_{\alpha, 1} m ~+~ m \cdot_{1, \alpha} a, \\
		\big(\delta^M_{\Alg} (f) \big)_{\alpha_1, \ldots, \alpha_{n+1}} (a_1, \ldots, a_{n+1} ) =~& (-1)^{n+1}a_1 \cdot_{\alpha_1 , \alpha_2 \dots \alpha_{n+1}} f_{\alpha_2, \ldots, \alpha_{n+1}} (a_2, \ldots, a_{n+1})  \\
		+ \sum_{i=1}^n (-1)^{n-i+1} ~ f_{\alpha_1, \ldots, \alpha_i \alpha_{i+1}, \ldots, \alpha_{n+1}} & \big(  a_1, \ldots, a_{i-1}, a_i \cdot_{\alpha_i,\, \alpha_{i+1}} a_{i+1}, a_{i+2}, \ldots, a_{n+1}  \big) \nonumber \\
		+ ~ f_{\alpha_1, \ldots, \alpha_n} &(a_1, \ldots, a_n) ~\cdot_{\alpha_1 \dots a_n,\, \alpha_{n+1}} a_{n+1}, \nonumber
	\end{align*}
	for $f \in \operatorname{Hom}_\Omega\left(A^{\otimes n}, M\right)$.
	The cohomology of $\mathcal{C}_{\Alg}^{\bullet}(A,\, M)$ is called {\bf the  Hochschild cohomology of  the $\Omega$-associaltive algebra $A$ with  coefficients in $M$ } and it is denoted by $\mathrm{H}_{\Alg}^{\bullet}(A,\, M)$. When the bimodule $M$ is taken to be $A$ itself, i.e. the regular $A$-bimodule $ A $, the cochain complex $\mathcal{C}^\bullet(A,A)$  is just called the Hochschild cochain complex of the $\Omega$-associative algebra $A$ and it is denoted as $\mathcal{C}^\bullet(A)$. The cohomology of $\mathcal{C}^\bullet(A)$, denoted by $\mathrm{H}^\bullet_{\rm Alg}(A)$, is just called\textbf{ the Hochschild cohomology of the $\Omega$-associaltive algebra $A$}.
\end{defn}

%	When the bimodule $M$ is taken to be $A$, i.e. the regular $A$-bimodule $ A $,
%	$M=A$,
%	the coboundary operator can be denoted by $\delta_{\Alg}(f) :=[\mu, f]_{\Omega},\,\text { for all }\, f \in \operatorname{Hom}_\Omega(A^{\otimes n}, A ).$ The

\smallskip
\section{$L_{\infty}[1]$-structure and cohomology theory of relative $\Omega$-family Rota-Baxter algebras}
\label{sec:relative cohomology}

In this section, we apply Voronov's higher derived brackets method ~\cite{Voronov} to construct the $L_{\infty}[1]$-algebra that characterizes relative $\Omega$-family Rota-Baxter algebras of weight $ \lambda $ as Maurer-Cartan elements. We obtain twisted $L_{\infty}[1]$-algebra that controls deformations of the relative $\Omega$-family Rota-Baxter algebra of weight $ \lambda $. Consequently, we define a cohomology theory of relative $\Omega$-family Rota-Baxter algebras of weight $ \lambda $ induced by the twisted $L_{\infty}[1]$-algebra.
\subsection{$L_{\infty}[1]$-algebras associated with relative $\Omega$-family Rota-Baxter algebras}\
\label{sub:linfty} 

In this subsection, we mainly construct the $L_\infty[1]$-structure associated with relative $\Omega$-family Rota-Baxter algebras of weight $\lambda $.

Firstly, let's introduce the definition of relative $\Omega$-family Rota-Baxter algebras of weight $\lambda$ and some related notions.
\begin{defn}
	An \textbf{AssAct} is a quintuple
	$(A, \mu=\{\mu_{\alpha,\,\beta}\}_{\alpha,\,\beta\in\Omega},
	V,
	\mu_V=\{\mu_{V,\,\alpha,\,\beta}\}_{\alpha,\,\beta\in\Omega},
	l=\{l_{\alpha,\,\beta}\}_{\alpha,\,\beta\in\Omega}
	, r
	=\{r_{\alpha,\,\beta}\}_{\alpha,\,\beta\in\Omega})$, where$(A,\mu) $ and $(V,\mu_{V})$ are $\Omega$-associative algebras and $ V $ is an $ A $-bimodule with the action $ l: A \ot V \to V$, $ r: V \ot A \to V $ , and these structures are compatible in the following sense:
	\begin{eqnarray*}
		\mu_{V,\,\alpha\beta,\,\gamma} \circ (l_{\alpha,\,\beta} \ot \id) &=& l_{\alpha,\,\beta\gamma} \circ (\id \ot \mu_{V,\,\beta,\,\gamma}),\\
		\mu_{V,\,\alpha\beta,,\gamma} \circ (r_{\alpha,\,\beta} \ot \id) &=& \mu_{V,\,\alpha,\,\beta\gamma} \circ (\id \ot l_{\beta,\,\gamma}),\\
		r_{\alpha\beta,\,\gamma} \circ (\mu_{V,\,\alpha,\,\beta} \ot \id) &=& \mu_{V,\,\alpha,\,\beta\gamma} \circ (\id \ot r_{\beta,\,\gamma}).
	\end{eqnarray*}
\end{defn}

\begin{defn}
	Let $(A, \mu,V, \mu_V,l,r)$ be an AssAct. A family of linear maps  $ T = \{T_\omega\}_{\omega\in\Omega}: V \to A $ is called a \textbf{relative $\Omega$-family Rota-Baxter operator of weight $\lambda$  } if
	\[ \mu_{\alpha, \, \beta} \circ (T_\alpha \ot T_\beta) = T_{\alpha\beta} \circ \big( l_{\alpha, \, \beta} \circ (T_\alpha \ot \id) + r_{\alpha, \, \beta} \circ (\id \ot T_\beta) + \lambda {\mu_{V}}_{\alpha, \, \beta}\big). \]
	In this case, we call
	$ (A,\mu,V,\mu_{V},l,r,T) $
	a \textbf{ relative $\Omega$-family Rota-Baxter algebra of weight $\lambda $ }.
\end{defn}

\begin{prop}\label{V-data RB}
	Let $ A $ and $ V $ be two vector spaces. We have a $V$-data $(\mathfrak{L}, \mathfrak{a}, P, \Delta)$ as follows:
	%\begin{itemize}
\begin{enumerate}
		\item  the graded Lie algebra $(\mathfrak{L},[-, -])$ is given by $\Big(\bigoplus\limits_{n=0}^{+\infty} \mathrm{Hom}_\Omega((A \oplus V)^{\ot n+1}, A \oplus V),[-, -]_{\Omega}\Big)$;
		\item  the abelian graded Lie subalgebra $\mathfrak{a}$ is given by
		$
		\mathfrak{a}:=\bigoplus\limits_{n=0}^{+\infty} \operatorname{Hom}_\Omega\big(V^{\ot n+1}, A\big)
		$;
		\item  $P: \mathfrak{L} \rightarrow \mathfrak{L}$ is the projection onto the subspace $\mathfrak{a}$;
		\item  $\Delta=0$.
	%\end{itemize}
\end{enumerate}
\end{prop}

Consider the following graded subspace of $\mathfrak{L}$:
\[\mathfrak{L}'
= \bigoplus\limits_{n=0}^{+\infty} \mathfrak{L}'(n+1)
:= \bigoplus\limits_{n=0}^{+\infty} \left(\mathrm{Hom}_\Omega(A^{\ot n+1},A) \bigoplus \bigoplus\limits_{i=0}^{n} \mathrm{Hom}_\Omega(\mathcal{A}^{i,n+1-i},V)\right),\]
where $\mathcal{A}^{p,q}$ is the subspace of $(A\oplus V)^{\otimes p+q}$ consisting of the tensor powers of $A$ and $V$ with $A$, $V$ appearing $p,q$ times respectively.
Obviously, $\mathfrak{L}'$ is a graded Lie subalgebra of $ \mathfrak{L} $ and it will determines an $L_\infty[1]$-algebra in the following way:
\begin{prop}\label{L infty classical}
	With the above notations, $ (s^{-1}\mathfrak{L}' \oplus \mathfrak{a}, \{l_{i}\}_{i=1}^{+\infty}) $ is an $ L_{\infty}[1] $-algebra, where
	\begin{eqnarray*}
		&& l_{1} = 0, \\
		&& l_{2}(s^{-1}f,s^{-1}g) = (-1)^{|f|} s^{-1}[f,g]_{\Omega},\\
		&& l_{i}(s^{-1}f, \theta_{1},\dots,\theta_{i-1}) = P\Big[\dots \big[[f,\theta_{1}]_{\Omega},\theta_{2}\big]_{\Omega} \dots, \theta_{i-1}\Big]_{\Omega}, i \geqslant 2,
	\end{eqnarray*}
	for homogeneous elements $ f,g \in \mathfrak{L}' $, $ \theta_{1}, \dots , \theta_{i-1} \in \mathfrak{a} $, and all the other $L_{\infty}[1]$-algebra products that are not obtained from the ones written above by permutations of arguments, will vanish.
\end{prop}

\begin{proof}
 It can be obtained immediately from Remark~\ref{rk:subalgebra}.
%  and Proposition~\ref{V-data RB}.
\end{proof}

To study the cohomology of relative $\Omega$-family Rota-Baxter algebras of weight $\lambda$, we modify the above $ L_{\infty}[1] $-algebra as follows:

\begin{thm}\label{modified L infty}
	Let $ \lambda \in \bfk$, with the above notations, $ (s^{-1}\mathfrak{L}' \oplus \mathfrak{a}, \{l'_{i}\}_{i=1}^{+\infty}) $ is an $ L_{\infty}[1] $-algebra, where
	\begin{eqnarray*}
		&& l'_{1} = 0, \\
		&& l'_{2}(s^{-1}f,s^{-1}g) = (-1)^{|f|} s^{-1}[f,g]_{\Omega},\\
		&& l'_{i}(s^{-1}f, \theta_{1},\dots,\theta_{i-1}) = \lambda^{n-(i-1)}
		P\Big[\dots \big[[f,\theta_{1}]_{\Omega},\theta_{2}\big]_{\Omega} \dots, \theta_{i-1}\Big]_{\Omega}
		,  i \geqslant 2,
	\end{eqnarray*}
	for homogeneous elements $ f\in \mathfrak{L}'(n)$ with $ |f| = n-1 $, $ g\in \mathfrak{L}'$, $ \theta_{1}, \dots , \theta_{i-1} \in \mathfrak{a} $, and all the other $L_{\infty}[1]$-algebra products that are not obtained from the ones written above by permutations of arguments, will vanish.
\end{thm}

\begin{proof}
	By Proposition \ref{L infty classical}, $ (s^{-1}\mathfrak{L}' \oplus \mathfrak{a} , \{l_{i}\}_{i=1}^{+\infty} )$ is an $ L_{\infty}[1] $-algebra. For homogeneous elements $ f \in \mathfrak{L}'(n) , g \in \mathfrak{L}' $ and $ \theta_{1} , \dots , \theta_{i-1} \in \mathfrak{a}$, we have
	\begin{eqnarray}
		&& l'_{1} = l_{1} = 0, \label{classical and modified 1}\\
		&& l'_{2}(s^{-1}f,s^{-1}g) = l_{2}(s^{-1}f,s^{-1}g), \label{classical and modified 2}\\
		&& l'_{i}(s^{-1}f, \theta_{1},\dots,\theta_{i-1}) = \lambda^{n-(i-1)} l_{i}(s^{-1}f, \theta_{1},\dots,\theta_{i-1}) . \label{classical and modified 3}
	\end{eqnarray}
	So $ \{l'_{i}\}_{i=1}^{+\infty} $ inherit the graded symmetry property from $ \{l_{i}\}_{i=1}^{+\infty} $. Now, let's check that the family $\{l_i'\}_{i=1}^{+\infty}$ satisfies the generalised Jacobi identity, i.e., for homogeneous elements, $x_1,\dots, x_n\in s^{-1}\mathfrak{L}'\oplus \mathfrak{a} $, the following equation holds:
	\begin{eqnarray}	\sum_{i=1}^n \sum_{\sigma \in \operatorname{Sh}(i, n-i)} \varepsilon(\sigma) l'_{n-i+1}\left(l'_i\left(x_{\sigma(1)}, \ldots, x_{\sigma(i)}\right), x_{\sigma(i+1)}, \ldots, x_{\sigma(n)}\right)=0 \label{Eq: modified-Jacobi-identity}.
		\end{eqnarray}
	
	By the definition of $\{l_i'\}_{i=1}^{+\infty}$, all terms in the Equation \eqref{Eq: modified-Jacobi-identity} above will be trivial except the following two cases:
	\begin{itemize}
		\item[(i)] when $n=3$, $x_1,x_2,x_3$ all come from $s^{-1}\mathfrak{L'}$.
		\item[(ii)] when $n\geq 4$, there are exactly two elements of $x_1,\dots, x_n$ belong to $s^{-1}\mathfrak{L}'$ and the rest elements belong to $\mathfrak{a}$. Assume that the two elements come from $s^{-1}\mathfrak{L}'(k_{1}), s^{-1}\mathfrak{L}'(k_{2})$
%		have degree $k_1, k_2$
		respectively.
		\end{itemize}
In case $\rm (i)$, the Eq.~\eqref{Eq: modified-Jacobi-identity} holds according to Eq. \eqref{classical and modified 2}.
In case $\rm (ii)$, by the definition of $\{l_i'\}_{i=1}^{+\infty}$, we have the following equality holds:
\[l'_{n-i+1}\left(l'_i\left(x_{\sigma(1)}, \ldots, x_{\sigma(i)}\right), x_{\sigma(i+1)}, \ldots, x_{\sigma(n)}\right)={\lambda^{k_{1}+k_{2}-n+1}}l_{n-i+1}\left(l_i\left(x_{\sigma(1)}, \ldots, x_{\sigma(i)}\right), x_{\sigma(i+1)}, \ldots, x_{\sigma(n)}\right).\] So Eq. \eqref{Eq: modified-Jacobi-identity} also holds in this case.

	In conclusion, $ (s^{-1}\mathfrak{L}' \oplus \mathfrak{a} , \{l'_{i}\}_{i=1}^{+\infty} )$ is an $ L_{\infty}[1] $-algebra.
\end{proof}
	
\begin{thm}\label{L infty MC and RB}
	Let $ A $ and $ V $ be two vector spaces endowed with five families of linear maps :
 %Consider $\mu=(\mu_{\alpha,\,\beta})_{\alpha,\,\beta\in\Omega}$ and $\mu_V=(\mu_{V,\,\alpha,\,\beta})_{\alpha,\,\beta\in\Omega}$ as follows:
	\begin{align*}
		&\mu: A \ot A \to A,\quad l:A \ot V \to V,\quad r: V \ot A \to V,\\
 &\mu_V: V \ot V \to V,\quad T: V \to A.
	\end{align*}
	Denote $ \pi = \mu + l + r +\mu_{V} $. Then $ (s^{-1}\pi, T)$
	is a Maurer-Cartan element of $ (s^{-1}\mathfrak{L}' \oplus \mathfrak{a}, \{l'_{i}\}_{i=1}^{+\infty}) $,
 %the $ L_{\infty}[1] $-algebra given by Theorem \ref{modified L infty},
	if and only if,  $ (A,\mu,V,\mu_{V},l,r,T) $ is a relative $\Omega$-family Rota-Baxter algebra of weight $\lambda$.
\end{thm}
\begin{proof}
	%	According to the propertiy of bidegree, we have
Notice that $ \Big[\big[[\pi,T]_{\Omega},T\big]_{\Omega},T\Big]_{\Omega} = 0$, so we have $l_n'((s^{-1}\pi, T)^{\ot n})=0$ for all $n\geqslant 4$. Then the Maurer-Cartan equation for the element $(s^{-1}\pi, T)$ is
	\begin{eqnarray*}
		&& \sum_{k=1}^{+\infty} \frac{1}{k !} l'_{k}((s^{-1}\pi,T),\dots,(s^{-1}\pi,T))\\
		&=& \frac{1}{2 !} l'_{2}((s^{-1}\pi,T),(s^{-1}\pi,T)) +
		\frac{1}{3 !} l'_{3}((s^{-1}\pi,T),(s^{-1}\pi,T),(s^{-1}\pi,T))\\
		&=& \frac{1}{2} l'_{2}(s^{-1}\pi,s^{-1}\pi) + l'_{2}(s^{-1}\pi,T) +\frac{1}{2} l'_{3}(s^{-1}\pi,T,T)\\
		&=& \Big(-\frac{1}{2}s^{-1}[\pi,\pi]_{\Omega},\lambda P[\pi,T]_{\Omega}+\frac{1}{2} P\big[\left[\pi,T\right]_{\Omega},T\big]_{\Omega}\Big).
	\end{eqnarray*}
	Note that
	\begin{eqnarray*}
		&&[\pi,\pi]_{\Omega}\\
		&=&[\mu,\mu]_{\Omega}+2[\mu,l]_{\Omega}+2[\mu,r]_{\Omega}+[l,l]_{\Omega}\\
         &&+2[l,r]_{\Omega}+2[l,\mu_{V}]_{\Omega}+[r,r]_{\Omega}+2[r,\mu_{V}]_{\Omega}+[\mu_{V},\mu_{V}]_{\Omega}\\
		&=&2(\mu \circ (\mu  \ot \id) - \{\mu \circ (\id \ot \mu))
		+2(l \circ (\mu \ot \id) - l \circ (\id \ot l))
		+2(r \circ (r \ot \id) - r \circ (\id \ot \mu))\\
		&&+2(r \circ (l \ot \id) - l \circ (\id \ot r))
		+2(\mu_{V} \circ (l \ot \id) - l \circ (\id \ot \mu_{V}))
		+2(\mu_{V} \circ (r \ot \id) - \mu_{V} \circ (\id \ot l))\\
		&&+2(r \circ (\mu_{V} \ot \id) - \mu_{V} \circ (\id \ot r))
		+2(\mu_{V} \circ (\mu_{V} \ot \id) - \mu_{V} \circ (\id \ot \mu_{V})),
	\end{eqnarray*}
	so
	\begin{eqnarray*}
		&&[\pi,\pi]_{\Omega}=0\\
		&\Leftrightarrow& \mu \circ (\mu \ot \id) = \mu \circ (\id \ot \mu),
		\hspace{1cm} l \circ (\mu \ot \id) = l \circ (\id \ot l),
		\hspace{1cm} r \circ (r \ot \id) = r \circ (\id \ot \mu),\\
		&&r \circ (l \ot \id) = l \circ (\id \ot r),
		\hspace{1.5cm} \mu_{V} \circ (l \ot \id) = l \circ (\id \ot \mu_{V}),
		\hspace{0.5cm} \mu_{V} \circ (r \ot \id) = \mu_{V} \circ (\id \ot l),\\
		&&r \circ (\mu_{V} \ot \id) = \mu_{V} \circ (\id \ot r),
		\quad \mu_{V} \circ (\mu_{V} \ot \id) = \mu_{V} \circ (\id \ot \mu_{V}),\\
		&\Leftrightarrow&  (A,\mu,V,\mu_{V},l,r)  \text{ is an AssAct.}
	\end{eqnarray*}
	Moreover, we have
	\begin{eqnarray*}
		&&\lambda P[\pi,T]_{\Omega}+\frac{1}{2} P\big[[\pi,T]_{\Omega},T\big]_{\Omega}\\
		&=&-\lambda T \circ \mu_{V} + \mu \circ (T \ot T) - T \circ l \circ (T \ot \id)- T \circ r \circ (\id \ot T).
%\\		&=&\{(-\lambda T \circ \mu_{V} + \mu \circ (T \ot T) - T \circ l \circ (T \ot \id)- T \circ r \circ (\id \ot T))_{\alpha, \beta }\}_{\alpha,\beta \in \Omega}
%		
% \\ &=& \{-\lambda T_{\alpha\beta} \circ {\mu_{V}}_{\alpha,\beta}  + \mu _{\alpha\beta} \circ (T_\alpha \ot T_\beta) - T_{\alpha\beta} \circ l_{\alpha\beta} \circ (T_\alpha \ot \id) - T_{\alpha\beta} \circ r_{\alpha\beta} \circ (\id \ot T_\beta) \}_{\alpha,\beta \in \Omega}
	\end{eqnarray*}
	Hence, $ (s^{-1}\pi, T)$ is a Maurer-Cartan element of $ (s^{-1}\mathfrak{L}' \oplus \mathfrak{a}, \{l'_{i}\}_{i=1}^{+\infty})$,  if and only if,
 $ (A,\mu,V,\mu_{V},$ $l,r,T) $ is a relative $\Omega$-family Rota-Baxter algebra of weight $\lambda$.
	%\[ [\pi,\pi]_{\Omega}=0 \text{ and } \lambda P[\pi,T]_{\Omega}+\frac{1}{2} P\left[\left[\pi,T\right]_{\Omega},T\right]_{\Omega} = 0.\]
\end{proof}

Then we obtain the twisted $L_\infty[1]$-algebra that controls the deformations of $\Omega$-family Rota-Baxter algebras of weight $\lambda $.

\begin{thm} \label{twisted L infty deformation}
With the  above notations, the twisted $L_\infty[1]$-algebra $(s^{-1}\mathfrak{L}'\oplus \frak a, \{l_{i}'^{(s^{-1}\pi, T)}\}_{i=1}^{+\infty})$ controls the deformations of the $\Omega$-associative Rota-Baxter algebra  $(A, \mu, V, \mu_V, l, r, T)$ of weight $\lambda$. Explicitly, $(A, \mu+\widehat{\mu}, V, \mu_V+\widehat{\mu}_V, l+\widehat{l}, r+\widehat{r}, T+\widehat{T})$ is again a relative $\Omega$-family Rota-Baxter algebra of weight $\lambda $ if and only if $(s^{-1}\widehat{\pi}, \widehat{T})$ is a Maurer-Cartan element of $(s^{-1}\mathfrak{L}'\oplus\frak a, \{l_{i}'^{(s^{-1}\pi, T)}\}_{i=1}^{+\infty})$, where $ \widehat{\pi} = \widehat{\mu} + \widehat{l} + \widehat{r} + \widehat{\mu_{V}}$.
\end{thm}

\begin{proof}
This is the direct corollary of Proposition \ref{twisted L infty}.
\end{proof}
\subsection{Cohomology theory of relative $\Omega$-family Rota-Baxter algebras}\
\label{sub:cohomology}

In Subsection~\ref{sub:linfty}, we have seen that the Maurer-Cartan elements of the $L_\infty[1]$-algebra $ (s^{-1}\mathfrak{L} \oplus \mathfrak{a}, \{l'_{i}\}_{i=1}^{+\infty})$ correspond to relative $\Omega$-family Rota-Baxter algebras of weight $\lambda$ on the space $A \oplus V$. Then we can deduce the cohomology theory of relative $\Omega$-family Rota-Baxter algebras of weight $\lambda $ from the twisting procedures in the $L_\infty[1]$-algebra $(s^{-1}\mathfrak{L}'\oplus \frak a, \{l_{i}'^{(s^{-1}\pi, T)}\}_{i=1}^{+\infty})$.

\begin{defn}
	Let $(A, \mu, V, \mu_V, l, r, T)$ be a relative $\Omega$-family Rota-Baxter algebra of weight $\lambda$. Define  \textbf{the cochain complex of the $(A,\mu, V,\mu_V,l,r,T)$} to be the cochain complex $(s\frak L'\oplus s^2\frak a, \partial)$ with $\partial=s^2{l'}^{(s^{-1}\pi,T)}_1s^{-2}$   and denote by  $\mathcal{C}^\bullet_{\rm{RelRBA_\lambda}}(A,V)$.
	The cohomology $\mathrm{H}^{\bullet}_{\rm{RelRBA_\lambda}}(A,V)$ of the cochain complex is called  \textbf{the cohomology of the  $\Omega$-family Rota-Baxter algebra $(A,\mu, V,\mu_V, l,r)$ of weight $\lambda$}.
\end{defn}

Let's give a specific description of the cochain complex.
For any
\[ f \in \Hom(A^{\ot n}, A) \bigoplus \bigoplus\limits_{i=0}^{n-1}\Hom(\mathcal{A}^{i, n-i}, V)\subset \frak L' \text{ and } \theta\in\Hom(V^{\ot n-1}, A)\subset \frak a,\]
%$f
%\in \Hom(A^{\ot n}, A)\oplus \oplus_{i=0}^{n-1}\Hom(\mathcal{A}^{i, n-i}, V)\subset \frak L'$ and $\theta\in\Hom(V^{\ot n-1}, A)\subset \frak a$,
we have
\begin{align*}
l'^{(s^{-1}\pi, T)}_1(s^{-1}f, \theta)
=& \ \sum_{k=0}^{+\infty}\frac{1}{k!}l'_{k+1}\Big(\underbrace{(s^{-1}\pi, T),\ldots,(s^{-1}\pi, T)}_{k}, (s^{-1}f, \theta)\Big)\\
=& \ l'_2\Big((s^{-1}\pi, T), (s^{-1}f, \theta)\Big)+\frac{1}{2}l'_3\Big((s^{-1}\pi, T), (s^{-1}\pi, T), (s^{-1}f, \theta)\Big)\\
& \ +\sum_{k=3}^{+\infty}\frac{1}{k!}l'_{k+1}\Big(\underbrace{(s^{-1}\pi, T),\ldots,(s^{-1}\pi, T)}_k,(s^{-1}f, \theta)\Big)\\
%=& \ l'_2(s^{-1}\pi, s^{-1}f)+l'_2(s^{-1}\pi, \theta)+l'_2(s^{-1}f, T)+l'_3(s^{-1}\pi, T, \theta)+\frac{1}{2}l'_3(s^{-1}f, T, T)\\
%& \ +\sum_{k=3}^{+\infty}\frac{1}{k!}l'_{k+1}(s^{-1}f, \underbrace{T,\ldots, T}_{k})+\frac{1}{n!}l'_{n+1}(s^{-1}f,\underbrace{T,\ldots, T}_n)\\
=& \ l'_2(s^{-1}\pi, s^{-1}f)+l'_2(s^{-1}\pi, \theta)+l'_3(s^{-1}\pi, T, \theta)
+\sum_{k=1}^n\frac{1}{k!}l'_{k+1}(s^{-1}f, \underbrace{T, \ldots, T}_k)\\
=& \ \Big(-s^{-1}[\pi, f]_{\Omega}, \lambda P[\pi, \theta]_{\Omega}+P\big[[\pi, T]_{\Omega}, \theta\big]_{\Omega} \ + \sum_{k=1}^n\frac{1}{k!}\lambda^{n-k}P\big[\ldots[f,\underbrace{T]_{\Omega},\ldots, T}_k\big]_{\Omega} \Big).
\end{align*}
In fact, we have
\begin{align*}
	\Big[\big[[\pi, T]_{\Omega}, T\big]_{\Omega}, \theta\Big]_{\Omega}=& \ 0,\\
	\Big[\ldots\big[[f, \underbrace{T]_{\Omega}, T\big]_{\Omega},\ldots, T\Big]_{\Omega}}_{k}=& \ 0,\,k\geqslant n+1.
\end{align*}

Write
\begin{align*}
	\delta_\pi(f)=& \ [\pi, f]_{\Omega},\\
	d_T(\theta)=& \ \lambda P[\pi, \theta]_{\Omega}+P\big[[\pi, T]_{\Omega}, \theta\big]_{\Omega}\\
	=&\ \lambda[\mu_V, \theta]_{\Omega}+\big[[\pi, T]_{\Omega}, \theta\big]_{\Omega},\\
	h_T(f)=& \  \sum_{k=1}^n\frac{1}{k!}\lambda^{n-k}P\big[\ldots[f, \underbrace{T]_{\Omega},\ldots, T}_k\big]_{\Omega}.
	%\ \sum_{k=1}^n\frac{1}{k!}l'_{k+1}(s^{-1}f, \underbrace{T,\ldots, T}_k)
\end{align*}
With the above notions, the coboundary operator on $\mathcal{C}^\bullet_{\rm RelRBA}(A,V)$ is give by
\[\partial(sf, s^2\theta)=
\big(-s\delta_\pi(f), s^{2}d_T(\theta)+s^2h_T(f)\big).\]
Notice that \[\delta_\pi(f)\in \frak L', d_T(\theta)\in \frak a, h_T(f)\in \frak a .\]
Thus we have $(s\frak L', \delta_\pi)$ and $(s\frak a, d_T)$ are also cochain complexes.
\begin{defn}
\begin{itemize}
	\item[(1)] The cochain complex $(s\frak L', \delta_\pi)$, denoted by $\mathcal{C}^\bullet_{\rm AssAct}(A,V)$, is called {\bf the cochain complex of the  AssAct} $(A,\mu, V,\mu_V, l,r)$. The cohomology $\rm H^\bullet_{AssAct}(A,V)$ of the cochain complex is called {{\bf the cohomology of the  AssAct} } $(A,\mu, V,\mu_V, l,r)$.
	\item[(2)] The cochain complex $(s\frak a , d_T)$, denoted by $\mathcal{C}^\bullet_{\rm RelRBO_\lambda}(T)$, is called {\bf the cochain complex of the relative $\Omega$-family Rota-Baxter operator $T$ of weight $\lambda$}. Its cohomology, denoted by $\rm H_{\rm RelRBO_\lambda}^\bullet(T)$, is called \textbf{the cohomology the relative $\Omega$-family Rota-Baxter operator $T$ of weight $\lambda$}.
%	=\{T_\omega\}_{\omega\in \Omega}
	\end{itemize}
\end{defn}

These three cochain complexes are combined in the following way.
\begin{prop}
The map $h_T$ defines a cochain map from $\mathcal{C}^\bullet_{\rm AssAct}(A,V)$ to $\mathcal{C}^\bullet_{\rm RelRBO_\lambda}(T)$, i.e., the following diagrams are commutative:
\[\begin{CD}%\label{dia:ext}
 		\ldots@>>> {\mathcal{C}_{\rm AssAct}^n(A,V)} @>\delta_\pi >> \mathcal{C}_{\rm AssAct}^{n+1}(A,V) @>\delta_\pi >> \mathcal{C}_{\rm AssAct}^{n+2}(A,V) @>>>\ldots\\
 		@. @V {h_T} VV @V {h_T} VV @V h_T VV @.\\
 		\ldots@>>> {\mathcal{C}_{\rm RelRBO_\lambda}^n(T)} @>d_T >> \mathcal{C}_{\rm RelRBO_\lambda}^{n+1}(T) @>d_T >> \mathcal{C}_{\rm RelRBO_\lambda}^{n+2}(T) @>>>\ldots.
 	\end{CD}\]
 And we have that $\mathcal{C}_{\rm RelRBA_\lambda}(A,V)\cong s \mathrm{Cone}(h_T)$.
 \end{prop}
 	\begin{proof}
 		This can be deduced from the definition of $\partial$  and  the fact that $\partial^2=0$.
 	\end{proof}
\begin{cor}
Let $(A, V,l,r,T)$ be a relative $\Omega$-family Rota-Baxter algebra  of weight $\lambda$. Then there is a short exact sequence of the cochain complexs:
\[\begin{CD}%\label{dia:ext}
 		0@>>> {s\mathcal{C}_{\rm RelRBO_\lambda}^{\bullet}(T)} @> \mathrm{inc} >> \mathcal{C}_{\rm RelRBA_\lambda}^{\bullet}(A,V) @> \mathrm{proj} >> \mathcal{C}_{\rm AssAct}^{\bullet}(A,V) @>>>0\\
 	\end{CD}\]
where $\mathrm{inc}$ and $\mathrm{proj}$ are the inclusion map and the projection map.

Consequently, there is a long exact sequence of the cohomology groups:
\[\begin{CD}%\label{dia:ext}
 		\ldots@>>> {\mathrm H _{\rm AssAct}^{n}(A,V)} @> >> \mathrm H_{\rm RelRBO_\lambda}^{n}(T) @> >> \mathrm H _{\rm RelRBA_\lambda}^{n+1}(A,V)@> >> \mathrm H _{\rm AssAct}^{n+1}(A,V) @>>>\ldots\\
 	\end{CD}\]
\end{cor}

\begin{remark} \label{remark appendix}
	As we have seen now, the complexes $\mathcal{C}^\bullet_{\rm AssAct}(A,V), \mathcal{C}^\bullet_{\rm RelRBA_\lambda}(A,V)$ and $\mathcal{C}^\bullet_{\rm RelRBO_\lambda}(T)$ contain many direct summands at each degree.  For later use, we need to describe  the coboundary operators $\delta_\pi$, $d_T$ and the cochain map $h_T$  on each component clearly. However, this process is  complicated, so we move it to the Appendix.
	%The coboundary operators $\delta_\pi, d_T$ and the cochain map $h_T$ are given explicitly and we move it to Appendix.
\end{remark}

\smallskip

\section{$L_{\infty}[1]$-structure and cohomology theory of absolute $\Omega$-family Rota-Baxter algebras}\
\label{sec:abosulte cohomology}
%In this section, we will define the cohomology of (absolute) $\Omega$-family Rota-Baxter algebras with the help of the general framework of the cohomology of regular relative $\Omega$-family Rota-Baxter algebras.
In this section, we will construct an $L_{\infty}[1]$-algebra structure associated with (absolute)  $\Omega$-family Rota-Baxter algebras  of weight $\lambda$. Then we will see that the twisted $L_\infty[1]$-algebra controls the deformations of the   $\Omega$-family Rota-Baxter algebras of weight $\lambda$, and the cohomology theory of $\Omega$-family Rota-Baxter algebras of weight $\lambda$ can be deduced naturally.
%In the end, we will define the cohomology of $\Omega$-family Rota-Baxter algebras with coefficients in bimodules induced by the general framework of the cohomology of $\Omega$-family Rota-Baxter algebras.

\begin{defn}
	Let $ (A, \mu = \{\mu_{\alpha,\,\beta}\}_{\alpha,\,\beta \in \Omega} )$ be an $\Omega$-associative algebra. A family of linear maps $  T=\{T_\omega\}_{\omega\in \Omega}: A \to A $ is called an (absolute) \textbf{$ \Omega $-family Rota-Baxter  operator of weight $ \lambda $} if it satifies
\[ \mu \circ (T \ot T) = T \circ \big( \mu \circ (T \ot \id) + \mu \circ (\id \ot T) + \lambda \mu \big). \]
i.e., for any $\alpha, \beta\in \Omega$, the following equation holds:
	\[ \mu_{\alpha,\,\beta} \circ (T_\alpha \ot T_\beta) = T_{\alpha\beta} \circ \big( \mu_{\alpha,\,\beta} \circ (T_\alpha \ot \id) + \mu_{\alpha,\,\beta} \circ (\id \ot T_\beta) + \lambda \mu_{\alpha,\,\beta}\big).\]
We call $(A,\mu , T)$ an (absolute)  \textbf{$\Omega$-family Rota-Baxter algebra of weight $ \lambda $}.
\end{defn}

\subsection{$L_{\infty}[1]$-algebras associated with absolute $\Omega$-family Rota-Baxter algebras}\

Now, let's introduce the $L_{\infty}[1] $-algebra,
  whose
Maurer-Cartan elements correspond to  $\Omega$-family Rota-Baxter algebras of weight $\lambda$ bijectively.

	We still use the symbol $ (s^{-1}\mathfrak{L}' \oplus \mathfrak{a}, \{l'_{i}\}_{i=1}^{+\infty}) $ to denote the $L_{\infty}[1] $-algebra structure given by Theorem \ref{modified L infty} when we take $ V $ to be $ A $. In this case,
	\begin{align*}
		\mathfrak{L}'
		& = \bigoplus\limits_{n=0}^{+\infty} \mathfrak{L}'(n+1)
		:= \bigoplus\limits_{n=0}^{+\infty} \left(\mathrm{Hom}_\Omega(A^{\ot n+1},A) \bigoplus \bigoplus\limits_{i=0}^{n} \mathrm{Hom}_\Omega(\mathcal{A}^{i,n+1-i},A)\right), \\ \mathfrak{a} & =\bigoplus\limits_{n=0}^{+\infty} \operatorname{Hom}_\Omega\big(A^{\ot n+1}, A\big).
	\end{align*}
	
	Denote
	\begin{align*}
		\mathfrak{M}(n+1)=&\Hom(A^{\ot n+1}, A),	\mathfrak{M} = \bigoplus\limits_{n=0}^{+\infty} \mathfrak{\mathfrak{M}}(n+1)   , \\
		\frak{h} =& \bigoplus\limits_{n=0}^{+\infty}\Hom(A^{\ot n+1}, A).
	\end{align*}
	Consider the embedding map
	\[  \iota: s^{-1}\mathfrak{M} \oplus \mathfrak{h} \to s^{-1} \mathfrak{L}' \oplus \mathfrak{a} \]
	defined by
%	\[ \iota(s^{-1}f) =  \widetilde{s^{-1}f}, \iota(\theta) = \theta \]
	\[ \iota(s^{-1}f) = s^{-1} \widetilde{f}, \iota(\theta) = \theta \]
	for $ f \in \Hom(A^{\ot n}, A) $, $ \theta \in \Hom(A^{\ot n-1}, A) $, where
%	\[ \widetilde{s^{-1} f}=\left(0, \ldots, 0, s^{-1} f ; s^{-1} f, s^{-1} f, \ldots, s^{-1} f, 0\right). \]
\[ \widetilde{f}=\left(0, \ldots, 0,  f ; f,  f, \ldots, f, 0\right) \in \mathfrak{L}(n). \]
%	\[ \overline{s^{-1}f} = (s^{-1}f, \dots, s^{-1}f) \in \Hom(A^{\ot n}, A)\bigoplus\bigoplus\limits_{i=0}^{n-1}\Hom(\mathcal{A}^{i,n-i}, A).\]
The sign convention on the right hand side of the above equation is defined in Eq.~\eqref{map decomposition} in Appendix.

Denote by $\mathrm{Im}(\iota)=\iota \left( s^{-1}\mathfrak{M} \oplus \mathfrak{h}\right)$. Then we have
\begin{prop} \label{sub L infty}
	$ (\mathrm{Im}(\iota),\{l'_{i}\}_{i=1}^{+\infty}) $ is an $L_{\infty}[1] $-subalgebra of $ (s^{-1}\mathfrak{L}' \oplus \mathfrak{a}, \{l'_{i}\}_{i=1}^{+\infty}) $.
\end{prop}
\begin{proof}
	We need to check that $\rm{Im}(\iota)$ is closed under the action of $l'_{i}$ for all $i\geqslant 1$.
For any $ 1 \leqslant m \leqslant n $, $ f \in \mathfrak{M}(n) $, $ g  \in \mathfrak{M}(m) $, $ \theta_{1}, \dots \theta_{n} \in \mathfrak{h} $, by Theorem~\ref{modified L infty} and  Eq.~\eqref{special f,g,G} in Appendix, we have
%\chao{You seem to have deleted my symbol convention, and I think that would be confusing here.}
\begin{align*}
	l_{2}'(\iota(s^{-1}f),\iota(s^{-1}g))
	&= l_{2}'(s^{-1}\widetilde{f},s^{-1}\widetilde{g}) \\
	&= (-1)^{|f|}s^{-1} [\widetilde{f},\widetilde{g}]_{\Omega} \\
	&= (-1)^{|f|}s^{-1} \left(0,\dots,0, [f,g]_{\Omega};  [f,g]_{\Omega} , \dots,   [f,g]_{\Omega},0\right) \\
	&= \iota \left((-1)^{|f|}s^{-1} [f,g]_{\Omega}\right),
\end{align*}
where $ |f|=n-1 $, and for $ i \geqslant 1 $,
\begin{align*}
	l'(\iota(s^{-1}f), \iota(\theta_{1}), \dots \iota(\theta_{i}))
	&= l'(s^{-1} \widetilde{f}, \theta_{1}, \dots \theta_{i})
	= \iota \left(l'(s^{-1} \widetilde{f}, \theta_{1}, \dots \theta_{i})\right) .
%	\\
%	&= \ddagger^{i}_{n} l'(\widetilde{\overline{s^{-1}f}}, \theta_{1}, \dots \theta_{n}) + \sum_{j=0}^{n-1} \ddagger^{i}_{j+1} l'(\widetilde{\overline{s^{-1}f}}_{j}, \theta_{1}, \dots \theta_{j+1}) \\
%	=& \ddagger^{i}_{n}  \Big[\dots \big[[f, \theta_1]_{\Omega},\theta_{2}\big] \dots, \theta_{n}\Big]_{\Omega}
%	+ \sum_{j=0}^{n-1}  \ddagger^{i}_{j+1}
%	\Big[\dots \big[[\tilde{f}_{j}, \theta_1]_{\Omega},\theta_{2}\big] \dots, \theta_{j+1}\Big]_{\Omega}
\end{align*}
\end{proof}

Let $p:\mathrm{Im}(\iota)\longrightarrow s^{-1}\mathfrak{M} \oplus \mathfrak{h}$  be the inverse map of $\iota: s^{-1}\mathfrak{M}\oplus \mathfrak{h}\rightarrow \rm{Im}(\iota)$, i.e., $ p(s^{-1} \widetilde{f}) = s^{-1}f$, $ p(\theta):=\theta $, for $ f \in \mathfrak{M}, \theta \in \mathfrak{h} $.
For $k\geqslant 1$, one can define
\[\varrho_{k}(x_1,\ldots,x_k):=p l'_k(\iota(x_1),\ldots,\iota(x_k))\]
for homogenous elements $x_{1},\dots,x_{k}\in s^{-1}\mathfrak{M}\oplus\frak{h}.$
 As a corollary of Proposition \ref{sub L infty},   we have
\begin{thm}
	$(s^{-1}\mathfrak{M}\oplus \frak h, \{\varrho_{i}\}^{+\infty}_{i=1})$ is an $L_\infty[1]$-algebra.
\end{thm}

Let us describe the $L_\infty[1]$-algebra structure on $s^{-1}\frakM\oplus \mathfrak{h}$  explicitly. Let $ f \in \mathfrak{M}(n) = \Hom(A^{\ot n} ,A )$, $ g  \in \mathfrak{M}(m) = \Hom(A^{\ot m} ,A ) $, $ \theta_{i}\in \mathfrak{h}(p_{i}) = \Hom(A^{\ot p_{i}} ,A ) $ with $m,n\geqslant 1$, $1 \leqslant i \leqslant n$. Then they have degrees  $ |f|=n-1, |g|=m-1, |\theta_{i}|=p_{i}-1 $ respectively. According to Proposition \ref{sub L infty} , we have
\begin{align*}
	\varrho_{1}
	=& 0, \\
	\varrho_{2}(s^{-1}f,s^{-1}g)
	=& (-1)^{|f|}s^{-1}[f,g]_{\Omega},\\
	\varrho_{n+1}(s^{-1}f,\theta_{1},\dots,\theta_{n})
	%	=& \Big[\dots \big[[f, \theta_1]_{G},\theta_{2}\big] \dots, \theta_{n}\Big]_{G},\\
	=& \sum_{\sigma \in \mathbb{S}_{n}}
	(-1)^{\xi_{1}}
	f \circ  (\theta_{\sigma(1)} \ot \dots \ot \theta_{\sigma(n)}) \\
	&+ \sum_{\sigma \in \mathbb{S}_{n}} \sum_{1\leqslant q_{2} < \dots <q_{n} \leqslant n} \sum_{r=1}^{p_{\sigma(1)}}
	(-1)^{\xi_{2}}
	\theta_{\sigma(1)} \circ_{r} f \circ  (\id \ot \dots \ot \underset{q_{2}}{\theta_{\sigma(2)}}
	\ot \dots
	%	\ot \id \ot \dots
	\ot \underset{q_{n}}{\theta_{\sigma(n)}} \dots \ot \id),\\
	\varrho_{k+1}(s^{-1}f,\theta_{1},\dots,\theta_{k})
	%	=& \lambda^{n-k} \Big[\dots \big[[f, \theta_1]_{G},\theta_{2}\big] \dots, \theta_{k}\Big]_{G}\\
	=& \sum_{\sigma \in \mathbb{S}_{k}}
	\sum_{1\leqslant q_{2} < \dots <q_{k} \leqslant n} \sum_{r=1}^{p_{\sigma(1)}}
	(-1)^{\xi_{3}} \lambda^{n-k}
	\theta_{\sigma(1)} \circ_{r} f \circ  (\id \ot \dots \ot \underset{q_{2}}{\theta_{\sigma(2)}}
	\ot \dots
	%	\ot \id \ot \dots
	\ot \underset{q_{k}}{\theta_{\sigma(k)}} \dots \ot \id),
\end{align*}
where
\begin{align*}
	(-1)^{\xi_{1}} =& \varepsilon\left(\sigma ; \theta_1, \ldots, \theta_n\right)
	(-1)^{\sum_{i=2}^{n}(\sum_{j=1}^{i-1}p_{\sigma (j)})|\theta_{\sigma(i)}| },\\
	(-1)^{\xi_{2}} =& \varepsilon\left(\sigma ; \theta_1, \ldots, \theta_n\right)
	(-1)^{
		1+|f||\theta_{\sigma(1)}|
		+  \sum_{i=2}^{n} (q_{i}+\sum_{j=2}^{i-1}p_{\sigma(j)}-i+1)|\theta_{\sigma(i)}| + (r-1)(|f| + \sum_{i=2}^{n}|\theta_{\sigma(i)}|)
	},\\
	(-1)^{\xi_{3}} =& \varepsilon\left(\sigma ; \theta_1, \ldots, \theta_k\right)
	(-1)^{
		1+|f| |\theta_{\sigma(1)}|
		+   \sum_{i=2}^{k} (q_{i}+\sum_{j=2}^{i-1}p_{\sigma(j)}-i+1)|\theta_{\sigma(i)}| + (r-1)(|f| + \sum_{i=2}^{k}|\theta_{\sigma(i)}|)
	}.
\end{align*}

Then we can see that the $L_\infty[1]$-algebra we just introduced can be used to describe the algebraic structre, $\Omega$-family Rota-Baxter algebras of weight $\lambda$, in the follwoing sense:
\begin{thm}\label{ab L infty}
Let $A$ be a vector space. Then the set of $\Omega$-family Rota-Baxter algebra structures of weight $\lambda$ on $A$ corresponds to the set of Maurer-Cartan elements in the $L_\infty[1]$-algebra $(s^{-1}\frakM\oplus \mathfrak{h},\{\varrho_i\}_{i=1}^{+\infty})$ bijectively.
\end{thm}
\begin{proof}
	Let $\alpha$ be an element of degree $0$ in $ (s^{-1}\mu, T) \in s^{-1}\mathfrak{M}\oplus \frak h$. Then $\alpha=(s^{-1}\mu,T)$ with $ \mu \in \Hom(A^{\ot 2},A) $ and $ T \in \Hom(A,A) $. Substituting $\alpha$ into the Maurer-Cartan equation in $(s^{-1}\mathfrak{M}\oplus \frak h, \{\varrho_{i}\}^{+\infty}_{i=1})$, we have
	\begin{align*}
		&\sum_{n=1}^{\infty} \frac{1}{n !} \varrho_{n}((s^{-1}\mu,T)) \\
		=& \frac{1}{2} \varrho_{2}(s^{-1}\mu,s^{-1}\mu) + \varrho_{2}(s^{-1}\mu,T) + \frac{1}{2} \varrho_{3}(s^{-1}\mu,T,T)\\
		=& \Big( -\frac{1}{2} s^{-1}[\mu,\mu]_{\Omega},
		- \lambda T \circ \mu
		+ \mu \circ (T \ot T)
		- T \circ \mu \circ (T \ot \id)
		- T \circ \mu \circ (\id \ot T) \Big).
	\end{align*}
 So $ (s^{-1}\mu , T) $ is a Maurer-Cartan element if and only if
 \begin{align*}
 	[\mu,\mu]_{\Omega} = 0 , \quad
 	\mu \circ (T \ot T) =
 	 T \circ \mu \circ (T \ot \id)
 	+ T \circ \mu \circ (\id \ot T)
 	+ \lambda T \circ \mu,
 \end{align*}
that is, $ (A,\mu,T) $ is an $\Omega$-family Rota-Baxter algebra of weight $ \lambda $.
\end{proof}

Let  $ (A,\mu,T) $ be an $\Omega$-family Rota-Baxter algebra of weight $\lambda$. By Theorem \ref{ab L infty}, we obtain that $ (s^{-1}\mu,T)$ is a Maurer-Cartan element in the $L_\infty[1]$-algebra $ (s^{-1}\mathfrak{M} \oplus \mathfrak{h}, \{\varrho_{i}\}_{i=1}^{+\infty}) $. Now we are ready to give the twisted $ L_{\infty}[1] $-algebra that controls deformatioms of the $\Omega$-family Rota-Baxter algebra  $(A,\mu,T)$ of weight $\lambda$.

\begin{thm} \label{ab twisted L infty deformation}
Let $(A,\mu,T)$ be an $\Omega$-family Rota-Baxter algebra of weight $\lambda$. The twisted $L_\infty[1]$-algebra  $(s^{-1}\mathfrak{M} \oplus \frak h, \{\varrho_{i}^{(s^{-1}\mu, T)}\}_{i=1}^{+\infty})$ induced by the Maurer-Cartan element $(s^{-1}\mu,T)$ controls the deformations of the $\Omega$-family Rota-Baxter algebra $(A,\mu,T)$ of weight $\lambda$. Explicitly, let $\widehat{\mu}:A \ot A\rightarrow  A, \widehat{T}:A\rightarrow A$ be two linear maps, then $(A, \mu+\widehat{\mu}, T+\widehat{T})$ is an $\Omega$-family Rota-Baxter algebra of  weight $\lambda$ if and only if $(s^{-1}\widehat{\mu}, \widehat{T})$ is a Maurer-Cartan element in the $L_\infty[1]$-algebra $(s^{-1}\mathfrak{M} \oplus \frak h, \{\varrho_{i}^{(s^{-1}\mu, T)}\}_{i=1}^{+\infty}).$
\end{thm}
\begin{proof}
	This is a direct corollary of Proposition \ref{twisted L infty} and Theorem \ref{ab L infty}.
\end{proof}

\begin{remark} In the previous construction, if we start with a general graded  vector space $A$ rather than an ordinary vector space, $s^{-1}\frak{M}\oplus \frak h$ will also be an $L_\infty[1]$-algebra. By solving the Maurer-Cartan equation in this $L_\infty[1]$-algebra, we will obtain the notion of homotopy $\Omega$-family Rota-Baxter algebras of weight $\lambda$. Modifying the methods in \cite{WZ2}, we can prove that the dg operad governing homotopy $\Omega$-family Rota-Baxter algebras of weight $\lambda$ is quasi-isomorphic to the operad of $\Omega$-family Rota-Baxter algebras of weight $\lambda$. Hence this dg operad should provide a model of the operad of $\Omega$-family Rota-Baxter algebras  of weight $\lambda$ in some suitable sense.
\end{remark}
%\smallskip

\subsection{Cohomology theory of absolute $\Omega$-family Rota-Baxter algebras}\

Using the twisting procedure in the $L_\infty[1]$-algebra associated with $\Omega$-family Rota-Baxter algebras of weight $\lambda$, we can define the cohomology of $\Omega$-family Rota-Baxter algebras  of weight $\lambda$ immediately. In the next Section, we will see that the formal deformations of  $\Omega$-family Rota-Baxter algebras of weight $\lambda$ are controlled by this cohomology theory.

\begin{defn}Let $(A,\mu,T)$ be an $\Omega$-family Rota-Baxter algebra of weight $ \lambda $. Then $(s^{-1}\mu, T)$ is a Maurer-Cartan element of $(s^{-1}\mathfrak{M} \oplus \frak h, \{\varrho_{i}\}_{i=1}^{+\infty})$. The cochain complex $(s\frakM\oplus s^2\frak h,s^2\varrho_1^{(s^{-1}\mu, T)}s^{-2})$, denoted by $(\mathcal{C}_{\rm RBA_\lambda}^\bullet(A,T), \partial_{\rm{RBA}_\lambda})$, is called \textbf{ the cochain complex of the $\Omega$-family Rota-Baxter algebra $(A,\mu,T)$  of weight $\lambda$}. The cohomology $\rm H^\bullet_{\rm RBA_\lambda}(A,T)$ of the cochain complex called \textbf{the cohomology of the $\Omega$-family Rota-Baxter algebra $(A,\mu,T)$ of weight $\lambda$}.
	\end{defn}

Let's describe $\mathcal{C}_{\rm RBA_\lambda}^\bullet(A,T)$ explicitly. Firstly, the component of  $\mathcal{C}_{\rm RBA_\lambda}^\bullet(A,T)$ on each degree is given as
\begin{align*}
	\mathcal{C}^0_{\rm{RBA}_\lambda}(A, T)=& \ 0,\\
	\mathcal{C}^1_{\rm{RBA}_\lambda}(A, T)=& \
	% \hm(A,A)=
	\mathcal{C}^1_{\rm{Alg}}(A),\\
	\mathcal{C}^n_{\rm{RBA}_\lambda}(A, T)=& \
	%\ \hm(A^{\ot n},A)\oplus\hm(A^{\ot n-1}, A)=
	\mathcal{C}^n_{\rm{Alg}}(A)\oplus \mathcal{C}^{n-1}_{\rm{RBO}_\lambda}(T),\,n\geqslant 2,
\end{align*}
where
\begin{eqnarray*}
\mathcal{C}^0_{\rm Alg}(A)=\mathcal{C}^{0}_{\rm{RBO}_\lambda}(T) = 0,	\mathcal{C}^n_{\rm Alg}(A)=\mathcal{C}^{n}_{\rm{RBO}_\lambda}(T) = \Hom(A^{\ot n},A), \quad n \geqslant 1.
	\end{eqnarray*}

 Let $f \in \Hom(A^{\ot n}, A) $ and $\theta\in\Hom(A^{\ot n-1}, A)$. By Proposition~\ref{twisted L infty}, the precise formula of $\varrho^{(s^{-1}\mu, T)}_1$ is given by the following:
\begin{align*}
	\varrho^{(s^{-1}\mu, T)}_1(s^{-1}f, \theta)
	=& \ \sum_{k=0}^{+\infty}\frac{1}{k!}\varrho_{k+1}\Big(\underbrace{(s^{-1}\mu, T),\ldots,(s^{-1}\mu, T)}_{k}, (s^{-1}f, \theta)\Big)\\
	=& \ \Big( \varrho_2(s^{-1}\mu, s^{-1}f) , \varrho_2(s^{-1}\mu, \theta)+\varrho_3(s^{-1}\mu, T, \theta)
	+\sum_{k=1}^n\frac{1}{k!}\varrho_{k+1}(s^{-1}f, \underbrace{T, \ldots, T}_k) \Big)\\
	=& \ \Big(-s^{-1} \delta_{\Alg}(f), \partial_{\RBO}(\theta) + \Phi(f)\Big),
\end{align*}
where
\begin{align*}
	\delta_{\Alg}(f) =& [\mu, f]_{\Omega} \\
	=& \mu\circ(f\ot\id)+(-1)^{n-1}\mu\circ(\id\ot f)
	+\sum_{i=1}^n(-1)^{n-i+1}f \circ_i \mu, \\
	\partial_{\RBO}(\theta)
	=&\ \varrho_2(s^{-1}\mu, \theta)+\varrho_3(s^{-1}\mu, T, \theta)\\
	=& \ (-1)^{n}\left(\mu\circ(T\ot\theta)-T\circ \mu \circ(\id\ot\theta)\right)+\sum_{i=1}^{n-1}(-1)^{n-i}\theta \circ_{i}
	( \mu \circ(T\ot\id)+ \mu \circ(\id\ot T)+\lambda \mu)
	%\circ(\id\ot\ldots\ot(l\circ(T\ot\id)+r\circ(\id\ot T)+\lambda\mu_V)\ot\ldots\ot\id)
	\\
	&+\mu\circ(\theta\ot T)-T\circ  \mu \circ(\theta\ot\id),\\
	\Phi(f)
	=&\sum_{k=1}^n\frac{1}{k!}\varrho_{k+1}(s^{-1}f, \underbrace{T, \ldots, T}_k)\\
	=& f\circ(\underbrace{T\ot\ldots\ot T}_{n})\\
	&+\sum_{k=0}^{n-1}\lambda^{n-k-1}\sum_{i_1,\ldots,i_{k+1}\geqslant 0}T\circ f\circ(\id^{\ot i_1}\ot T\ot\id^{\ot i_2}\ot\ldots\ot\id^{\ot i_k}\ot T\ot\id^{\ot^{i_{k+1}}}).
\end{align*}
Then the coboundary operator $\partial_{\rm{RBA}_\lambda}$ is the following:
\[\partial_{\rm{RBA}_\lambda}\big(sf,s^2\theta)=
\big(s\delta_{\rm{Alg}}(f), s^2(-\partial_{\rm{RBO}_\lambda}(\theta)-\Phi(f))\big),\,\forall
(f,\theta) \in \mathcal{C}^n_{\rm{RBA}_\lambda}(A, T).
\]
Notice that $(\mathcal{C}_{\rm RBO_\lambda}^{\bullet+1}(T), -\partial_{{\rm RBO}_\lambda})$ is a subcomplex of $(\mathcal{C}_{\rm RBA_\lambda}^\bullet(A,T),\partial_{\rm{RBA}_\lambda})$.
\begin{defn}\label{Definition: cochain complex of Rota-Baxter operator}
	The cochain complex $(\mathcal{C}^\bullet_{\rm RBO_\lambda}(T), \partial_{\mathrm{RBO}_\lambda})$ is called \textbf{the cochain complex of the $\Omega$-family Rota-Baxter operator $T $ of weight $\lambda$}. The cohomology of $(\mathcal{C}^\bullet_{\rm RBO_\lambda}(T), \partial_{\mathrm{RBO}_\lambda})$, denoted by $\rm H_{\rm RBO_\lambda}^\bullet(T)$, is called \textbf{the cohomology of the $\Omega$-family Rota-Baxter operator $T $ of weight $\lambda$}.
	\end{defn}

Similar with before, we have:
\begin{prop}\label{coro A complex}
	$\Phi $ is a cochain map from the Hochschild cochain complex $\big( \mathcal{C}_{\Alg}^\bullet(A), \delta_{\Alg} \big)$ to $\big( \mathcal{C}_{\mathrm{RBO}}^\bullet(T), \partial_{\rm{RBO}_\lambda} \big)$, i.e., satisfying the following commutative diagrams:
	\[\begin{CD}%\label{dia:ext}
		\ldots@>>> {\mathcal{C}_{\rm Alg}^n(A)} @>\delta_{\rm Alg} >> \mathcal{C}_{\rm Alg}^{n+1}(A) @>\delta_{\rm Alg} >> \mathcal{C}_{\rm Alg}^{n+2}(A) @>>>\ldots\\
		@. @V {\Phi} VV @V {\Phi} VV @V \Phi VV @.\\
		\ldots@>>> {\mathcal{C}_{\rm RBO_\lambda}^n(T)} @>\partial_{\rm RBO_\lambda} >> \mathcal{C}_{\rm RBO_\lambda}^{n+1}(T) @>\partial_{\rm RBO_\lambda} >> \mathcal{C}_{\rm RBO_\lambda}^{n+2}(T) @>>>\ldots.
	\end{CD}\]
		%. In other words
\end{prop}
\begin{proof}
	This is  a direct consequence of the fact that $\partial_{\rm{RBA}_\lambda}^2=0$.
\end{proof}

\begin{prop}
	Let $(A, \mu, T)$ be an $\Omega$-family Rota-Baxter algebra of weight $\lambda$. Then there is a short exact sequence of the cochain complexes:
	\[\begin{CD}%\label{dia:ext}
		0@>>> {s\mathcal{C}^{\bullet}_{\rm{RBO}_\lambda}(T)} @> \mathrm{inc} >> \mathcal{C}^{\bullet}_{\rm{RBA}_\lambda}(A,T) @> \mathrm{proj} >> \mathcal{C}^{\bullet}_{\rm{Alg}}(A) @>>>0\\
	\end{CD}\]
	where $\mathrm{inc}$ and $\mathrm{proj}$ are the inclusion map and the projection map.
	
	Consequently, there is a long exact sequence of the cohomology groups:
	\[\begin{CD}%\label{dia:ext}
		\ldots@>>> {\mathrm{H}^{n}_{\rm{Alg}}(A)} @> >> \mathrm{H}^{n}_{\rm{RBO}_\lambda}(T) @> >> \mathrm{H}^{n+1}_{\rm{RBA}_\lambda}(A,T)@> >> \mathrm{H}^{n+1}_{\rm{Alg}}(A) @>>>\ldots
	\end{CD}\]
\end{prop}

\smallskip

\section{Formal deformations and cohomological interpretations}~\label{sec:formal deformations}\

In this section, we mainly study the formal deformations of relative (resp. absolute) $\Omega$-family Rota-Baxter algebras of weight $\lambda$. And  we will see that the formal deformations of relative (resp. absolute) $\Omega$-family Rota-Baxter algebras  of weight $\lambda$ can be interpretated by the cohomologies of relative (resp. absolute) $\Omega$-family Rota-Baxter algebras  of weight $\lambda$ we have introduced before. And we obtain that the infinitesimals of equivalent formal deformations correspond to same cohomogy class. Thus, the cohomology theory we introduce here is the right cohomology theory for relative (resp. absolute) $\Omega$-family Rota-Baxter algebras of weight $\lambda$.

\subsection{Formal deformations of relative $\Omega$-family Rota-Baxter algebras}\

%Firstly, we consider the formal deformations of $\Omega$-family Rota-Baxter algebras of weight $\lambda$.

Let $(A, \mu, V, \mu_V, l, r, T)$ be a relative $\Omega$-family Rota-Baxter algebra of weight $\lambda$. Consider the space $A[[t]]$ (resp. $V[[t]])$ of formal power series in $t$ with coefficients in $A$ (resp. $V$). Then $A[[t]]$ and $V[[t]]$ are both $\bfk[[t]]$-modules.
\begin{defn}
	Let $(A, \mu, V, \mu_V, l, r, T)$
%	$(A, \mu = (\mu_{\alpha,\,\beta})_{\alpha,\,\beta \in \Omega}, V, \mu_V = ({\mu_{V}}_{\alpha,\,\beta})_{\alpha,\,\beta \in \Omega}, l = (l_{\alpha,\,\beta})_{\alpha,\,\beta \in \Omega}, r = (r_{\alpha,\,\beta}), T = (T_{\omega}))_{\omega \in \Omega}$
 be a relative $\Omega$-family Rota-Baxter algebra of weight $\lambda$. {\bf A one-parameter formal deformation} of $(A, \mu, V, \mu_V, l, r, T)$ consists of a tuple $(\mu_t, \mu_{V_t}, l_t, r_t, T_t)$ of formal power series of the form
	\begin{align*}
		\mu_t=\sum_{i\geqslant 0}\mu_it^i,\quad {\mu_{V}}_{t}=\sum_{i\geqslant 0}{\mu_{V}}_{i}t^i,\quad l_t=\sum_{i\geqslant 0}l_it^i,\quad r_t=\sum_{i\geqslant 0}r_it^i,\quad T_t=\sum_{i\geqslant 0}T_it^i,
	\end{align*}
	where $
	\mu_i = \{{\mu_i}_{\alpha, \, \beta}\}_{\alpha, \, \beta \in \Omega} \in \Hom(A \ot A , A),
	{\mu_{V}}_{i} = \{{{\mu_{V}}_{i}}_{\alpha, \, \beta}\}_{\alpha, \, \beta \in \Omega} \in \Hom(V\ot V,V),
	l_i = \{{l_i}_{\alpha, \, \beta}\}_{\alpha, \, \beta \in \Omega}\in\Hom(A\ot V, A),
	r_i = \{{r_i}_{\alpha, \, \beta}\}_{\alpha, \, \beta \in \Omega} \in \Hom(V\ot A, A),
	T_i = \{{T_i}_{\omega}\}_{\omega \in \Omega} \in \Hom(V, A)$ for all $i\geqslant 0$ with $\mu_0=\mu, l_0=l, r_0=r , {\mu_{V}}_{0} = {\mu_{V}}$ and $T_0=T$ such that
	$(A[[t]], \mu_t, V[[t]], {\mu_{V}}_{t}, l_t, r_t, T_t)$ is a relative $\Omega$-family Rota-Baxter algebra of weight $\lambda$ over the ring $\bfk[[t]].$
\end{defn}

It follows that the quintuple $(\mu_t, {\mu_{V}}_{t}, l_t, r_t, T_t)$ is a formal one-parameter deformation of relative $\Omega$-family Rota-Baxter algebra
%$(A,\mu = (\mu_{\alpha,\,\beta})_{\alpha,\,\beta \in \Omega}
%, T = (T_\omega)_{\omega\in\Omega})$
$(A, \mu, V, \mu_V, l, r, T)$ of weight $\lambda$ if and only if the following equations hold:
\begin{align*}
	\mu_t\circ(\mu_t\ot\id)= \  &\mu_t\circ(\id\ot\mu_t),\quad l_t\circ(\mu_t\ot\id)=l_t\circ(\id\ot l_t),\\
	r_t\circ(r_t\ot\id)=& \  r_t\circ(\id\ot\mu_t),\quad r_t\circ(l_t\ot\id)=l_t\circ(\id\ot r_t),\\
	{\mu_{V}}_{t}\circ(l_t\ot\id)=& \ l_t\circ(\id\ot{\mu_{V}}_{t}),\quad {\mu_{V}}_{t}\circ(r_t\ot\id)={\mu_{V}}_{t}\circ(\id\ot l_t),\\
	r_t\circ({\mu_{V}}_{t}\ot\id)=& \ {\mu_{V}}_{t}\circ(\id\ot r_t),\quad{\mu_{V}}_{t}\circ({\mu_{V}}_{t}\ot\id)={\mu_{V}}_{t}\circ(\id\ot{\mu_{V}}_{t}),\\
	\mu_t\circ(T_t\ot T_t)=& \ T_t\circ(l_t\circ(T_t\ot\id )+r_t\circ(\id\ot T_t)+\lambda {\mu_{V}}_{t}).
\end{align*}
%That is to say
Expanding these equations and comparing the coefficient of $t^n$, we obtain  that: for any $n\geqslant 0$,
\begin{align}
	\label{eq:deform1} \sum_{i=0}^n\mu_i\circ(\mu_{n-i}\ot\id)=&\sum_{i=0}^n\mu_i\circ(\id\ot\mu_{n-i}),
	\quad\sum_{i=0}^nl_i\circ(\mu_{n-i}
	\ot\id)=\sum_{i=0}^n l_i\circ(\id\ot l_{n-i}),\\
	\label{eq:deform2} \sum_{i=0}^n r_i\circ(r_{n-i}\ot\id)=&\sum_{i=0}^n r_i\circ(\id\ot \mu_{n-i}),\quad\sum_{i=0}^n r_i\circ(l_{n-i}\ot\id)=\sum_{i=0}^nl_i\circ(\id\ot r_{n-i}),\\
	\label{eq:deform3} \sum_{i=0}^n\mu_{V_i}\circ(l_{n-i}\ot\id)=&\sum_{i=0}^n l_i\circ(\id\ot{\mu_{V}}_{n-i}),\quad\sum_{i=0}^n{\mu_{V}}_i\circ(r_{n-i}\ot\id)=\sum_{i=0}^n{\mu_{V}}_i\circ(\id\ot l_{n-i}),\\
	\label{eq:deform4} \sum_{i=0}^n r_i\circ({\mu_{V}}_{n-i}\ot\id)=&\sum_{i=0}^n{\mu_{V}}_i\circ(\id\ot r_{{n-i}}),\quad
	\sum_{i=0}^n{\mu_{V}}_i\circ({\mu_{V}}_{n-i}\ot\id)=\sum_{i=0}^n{\mu_{V}}_i\circ(\id\ot {\mu_{V}}_{n-i}),\\
	\label{eq:deform5} \sum_{\substack{i+j+k=n\\ i, j,k\geqslant 0}}\mu_i\circ(T_j\ot T_k)=& \ \sum_{\substack{i+j+k=n\\ i, j,k\geqslant 0}}T_i\circ l_j\circ(T_k\ot \id)+\sum_{\substack{i+j+k=n\\ i, j,k\geqslant 0}}T_i\circ r_j\circ(\id\ot T_k)+\sum_{\substack{i+j=n\\ i, j\geqslant 0}}\lambda T_i\circ {\mu_{V}}_j.
\end{align}

Note that Eqs.~ \eqref{eq:deform1}-\eqref{eq:deform5} hold trivially for $n=0$ as $(A, \mu, V, \mu_V, l, r, T)$ is a relative $\Omega$-family Rota-Baxter algebra of weight $\lambda$.
Denote by$$ \pi_t
=\sum_{i\geqslant 0}\pi_it^i
:= \mu_{t} + l_{t} + r_{t} + {\mu_{V}}_{t},$$
that is $ \pi_{i} = \mu_{i} + l_{i} + r_{i} + {\mu_{V}}_{i} $. And the pair
%	$(
%	\mu_{1,\,\alpha,\,\beta}
%	(\pi_{1,\,\alpha,\,\beta})_{\alpha,\,\beta \in \Omega}
%	,(T_{1,\,\alpha})_{\alpha\in\Omega})$
$ (\pi_1, T_1) $ is called {\bf  the infinitesimal of the 1-parameter formal deformation}
$(A[[t]], \mu_t, V[[t]], {\mu_{V}}_{t}, l_t, r_t, T_t)$
%	$(A[[t]],
%	(\mu_{t,\,\alpha,\,\beta})_{\alpha,\,\beta \in \Omega}
%	,(T_{t,\,\omega})_{\omega\in\Omega})$
of the relative $\Omega$-family Rota-Baxter algebra $(A, \mu, V, \mu_V, l, r, T)$ of weight $\lambda$.

\begin{prop}
	Let $(A[[t]], \mu_t, V[[t]], {\mu_{V}}_{t}, l_t, r_t, T_t)$ be a 1-parameter formal deformation of the relative
	$\Omega$-family Rota-Baxter algebra  $(A, \mu, V, \mu_V, l, r, T)$
%	$(A,
%	%	  \mu_{\alpha,\,\beta}
%	(\mu_{\alpha,\,\beta})_{\alpha,\,\beta \in \Omega}
%	,(T_\omega)_{\omega\in\Omega})$
	of weight $\lambda$. Then the infinitesimal
	$(
	%	\mu_{1,\,\alpha,\,\beta}
	\pi_{1} = \{\pi_{1,\,\alpha,\,\beta}\}_{\alpha,\,\beta \in \Omega}
	,$ $T_{1} = \{T_{1,\,\omega}\}_{\omega\in\Omega})$ is a 2-cocycle in the cochain complex
	$\mathcal{C}_{\rm RelRBA_\lambda}^\bullet(A,,V)$.
\end{prop}
\begin{proof}
	When $n=1$, Eqs.~\eqref{eq:deform1}-\eqref{eq:deform4}
	become
	\begin{align*}
		\mu_{1} \circ (\mu \ot \id) + \mu \circ (\mu_{1} \ot \id) &= \mu_{1} \circ (\id \ot \mu) + \mu \circ (\id \ot \mu_{1}),\\
		l_{1} \circ (\mu \ot \id) + l \circ (\mu_{1} \ot \id) &= l_{1} \circ (\id \ot l) + l \circ (\id \ot l_{1}),\\
		r_{1} \circ (r \ot \id) + r \circ (r_{1} \ot \id) &= r_{1} \circ (\id \ot \mu) + r \circ (\id \ot \mu_{1}),\\
		r_{1} \circ (l \ot \id) + r \circ (l_{1} \ot \id) &= l_{1} \circ (\id \ot r) + l \circ (\id \ot r_{1}),\\
		{\mu_{V}}_{1} \circ (l \ot \id) + {\mu_{V}} \circ (l_{1} \ot \id) &= l_{1} \circ (\id \ot {\mu_{V}}) + l \circ (\id \ot {\mu_{V}}_{1}),\\
		{\mu_{V}}_{1} \circ (r \ot \id) + {\mu_{V}} \circ (r_{1} \ot \id) &= {\mu_{V}} \circ (\id \ot l) + {\mu_{V}} \circ (\id \ot l_{1}),\\
		r_{1} \circ ({\mu_{V}} \ot \id) + r \circ ({\mu_{V}} \ot \id) &= {\mu_{V}} \circ (\id \ot r) + {\mu_{V}} \circ (\id \ot r_{1}),\\
		{\mu_{V}}_{1} \circ ({\mu_{V}} \ot \id) + {\mu_{V}} \circ ({\mu_{V}} \ot \id) &= {\mu_{V}} \circ (\id \ot {\mu_{V}}) + {\mu_{V}} \circ (\id \ot {\mu_{V}}_{1}),
	\end{align*}
which
	imply that $\delta_\pi(\pi_1)=0$ by Eq.~\eqref{relative delta} in the Appendix. And Eq.~\eqref{eq:deform5} becomes
	\begin{align*}
		&\mu_1\circ(T\ot T)-(T\circ r_1\circ(\id\ot T)+T\circ l_1\circ(T\ot\id)+\lambda T\circ\mu_{V_1})\\
		=&-\Big(\mu\circ(T\ot T_1)-T \circ(\id\ot T_1)\Big)
		+\Big(T_1\circ r\circ(\id\ot T)+T_1\circ l\circ(T\ot\id)+\lambda T_1\circ\mu_V\Big)\\
		&-\Big(\mu\circ(T_1\ot T)-T\circ l\circ(T_1\ot\id)\Big),
	\end{align*}
	so we have
	$h_T(\pi_1)=-d_T(T_1)$ by Eqs.~\eqref{relative partial}-\eqref{relative chain map} in the Appendix. Thus $(\pi_1, T_1)$ is a $2$-cocycle in $ \mathcal{C}^{\bullet}_{\rm{RelRBA}_\lambda}(A,V) $.
\end{proof}

\subsection{Formal deformations of absolute $\Omega$-family Rota-Baxter algebras}\

In this subsection, we study formal deformations of absolute Rota-Baxter algebras of weight $\lambda$. Let $(A, \mu, T)$ be an $\Omega$-family Rota-Baxter algebra of weight $\lambda$. We consider the space $A[[t]]$ of formal power series in $t$ with coefficients in $A$. Then $A[[t]]$ is a $\bfk[[t]]$-modules.
\begin{defn}
	Let $(A, \mu, T)$
	be an $\Omega$-family Rota-Baxter algebra of weight $\lambda$. {\bf A one-parameter formal deformation} of $(A, \mu, T)$ consists of a pair $(\mu_t, T_t)$ of two formal power series of the form
	\begin{align*}
		\mu_t=\sum_{i\geqslant 0}\mu_it^i,\quad T_t=\sum_{i\geqslant 0}T_it^i,
	\end{align*}
	where $
	\mu_i = \{{\mu_i}_{\alpha, \, \beta}\}_{\alpha, \, \beta \in \Omega} \in \Hom(A \ot A , A),
	T_i = \{{T_i}_{\omega}\}_{\omega \in \Omega} \in \Hom(A, A)$ for all $i\geqslant 0$ with $\mu_0=\mu$ and $T_0=T$ such that
	$(A[[t]], \mu_t, T_t)$ is an $\Omega$-family Rota-Baxter algebra of weight $\lambda$ over the ring $\bfk[[t]].$
\end{defn}

Power series families $ \mu_{t} $ and
$ T_{t}$
determine a 1-parameter formal deformation of $\Omega$-family Rota-Baxter algebra $(A,\mu,T)$ of weight $\lambda$ if and only if
%for any $a,b,c\in A$,
the following equations hold
\begin{eqnarray*}
	\mu_{t,\,\alpha\beta,\,\gamma}(\mu_{t,\,\alpha,\,\beta} \ot \id) &=& \mu_{t,\,\alpha,\,\beta\gamma}( \id \ot \mu_{t,\,\beta,\,\gamma} ),\\
	\mu_{t,\,\alpha,\,\beta}(T_{t,\alpha} \ot T_{t,\,\beta} )&=& T_{t,\alpha\beta}\Big(\mu_{t,\,\alpha,\,\beta}(\id \ot T_{t,\,\beta} )+\mu_{t,\,\alpha,\,\beta}(T_{t,\alpha} \ot \id)+\lambda \mu_{t,\,\alpha,\,\beta} \Big).
\end{eqnarray*}
Expanding these equations and comparing the coefficient of $t^n$, we obtain  that
% $\{\mu_{i,\,\alpha,\,\beta}\}_{i\geqslant0,\,\alpha,\,\beta\in\Omega}$ and $\{T_{i,\,\omega}\}_{i\geqslant0,\,\omega\in\Omega}$
$\mu_{i}$ and $T_{i}$
have to  satisfy: for any $n\geqslant 0$,
\begin{equation}\label{Eq: deform eq for  products in RBA}
\sum_{i=0}^n\mu_{i,\,\alpha\beta,\, \gamma}\circ(\mu_{n-i,\,\alpha,\,\beta}\ot \id)
=\sum_{i=0}^n\mu_{i,\,\alpha,\,\beta \gamma}\circ(\id\ot \mu_{n-i,\,\beta,\, \gamma}),\end{equation}
\begin{equation}\label{Eq: Deform RB operator in RBA} \begin{array}{rcl}
	\sum\limits_{i+j+k=n\atop i, j, k\geqslant 0}	\mu_{i,\,\alpha,\,\beta}\circ(T_{j,\,\alpha}\ot T_{k,\,\beta})&=&\sum\limits_{i+j+k=n\atop i, j, k\geqslant 0} T_{i,\,\alpha\beta}\circ \mu_{j,\,\alpha,\,\beta}\circ (\id\ot T_{k,\,\beta})\\
	&  &+\sum\limits_{i+j+k=n\atop i, j, k\geqslant 0} T_{i,\,\alpha\beta}\circ\mu_{j,\,\alpha,\,\beta}\circ (T_{k,\,\alpha}\ot \id)+\lambda\sum\limits_{i+j=n\atop i, j \geqslant 0}T_{i,\,\alpha\beta}\circ\mu_{j,\,\alpha,\,\beta}.
\end{array}\end{equation}

Obviously, when $n=0$, the above conditions are exactly the associativity of $\mu$ and that $T$ is an $\Omega$-family Rota-Baxter operator of weight $\lambda$ with respect to $\mu$. The pair $(\mu_1 = \{{\mu_{1}}_{\alpha,\,\beta}\}_{\alpha,\,\beta \in \Omega}, T_1 = \{{T_1}_{\omega}\}_{\omega \in \Omega})$
is called {\bf  the infinitesimal of the 1-parameter formal deformation}  $(A[[t]], \mu_t, T_t)$
of  the $\Omega$-family Rota-Baxter algebra  $(A, \mu, T)$ of weight $\lambda$.

\begin{prop}\label{Prop: Infinitesimal is 2-cocyle}
Let
$(A[[t]],\mu_t, T_t)$
%$(A[[t]],
%(\mu_{t,\,\alpha,\,\beta})
%_{\alpha,\,\beta \in \Omega}
%,(T_{t,\,\omega})_{\omega\in\Omega})$
be a 1-parameter formal deformation of the
$\Omega$-family Rota-Baxter algebra
$(A,\mu,T)$
%$(A,(\mu_{\alpha,\,\beta})_{\alpha,\,\beta \in \Omega}
%,(T_\omega)_{\omega\in\Omega})$
of weight $\lambda$. Then the infinitesimal
$(\mu_1 = \{{\mu_{1}}_{\alpha,\,\beta}\}_{\alpha,\,\beta \in \Omega}, T_1 = \{{T_1}_{\omega}\}_{\omega \in \Omega})$
%$((\mu_{1,\,\alpha,\,\beta})_{\alpha,\,\beta \in \Omega}, (T_{1,\,\alpha})_{\alpha\in\Omega})$
is a 2-cocycle in the cochain complex
$\mathcal{C}_{\RBA}^\bullet(A,T)$.
\end{prop}
\begin{proof} When $n=1$, Eqs.~(\ref{Eq: deform eq for  products in RBA}) and (\ref{Eq: Deform RB operator in RBA})  become
	\[\begin{array}{rcl}
		\mu_{1,\,\alpha\beta,\,\gamma}\circ(\mu_{\alpha,\,\beta}\ot \id)
		+\mu_{\,\alpha\beta,\,\gamma}\circ(\mu_{1,\,\alpha,\,\beta}\ot \id)
		&=&  \mu_{1,\,\alpha,\,\beta \gamma}\circ(\id\ot \mu_{\beta,\, \gamma})
		+\mu_{\,\alpha,\,\beta \gamma}\circ (\id\ot \mu_{1,\,\beta,\, \gamma}),
	\end{array}\]
	and
	\[\begin{array}{cl}
		&\mu_{1,\,\alpha,\,\beta} (T_\alpha\ot T_\beta)
		-\left(T_{\alpha\beta}\circ\mu_{1,\,\alpha,\,\beta}\circ(\id\ot T_\beta)
		+T_{\alpha\beta}\circ\mu_{1,\,\alpha,\,\beta}\circ(T_\alpha\ot \id)
		+\lambda T_{\alpha\beta}\circ \mu_{1,\,\alpha,\,\beta}\right)\\
		=&
		-\left(\mu_{\alpha,\,\beta}\circ(T_\alpha\ot T_{1,\,\beta})-T_{\alpha\beta}\circ\mu_{\alpha,\,\beta}\circ(\id\ot T_{1,\,\beta})\right)
		-\left(\mu_{\alpha,\,\beta}\circ(T_{1,\,\alpha}\ot T_\beta)-T_{\alpha\beta}\circ\mu_{\alpha,\,\beta}\circ(T_{1,\,\alpha}\ot\id)\right)\\
		&+\left(T_{1,\alpha\beta}\circ\mu_{\alpha,\,\beta}\circ(\id\ot T_\beta)+T_{1,\,\alpha\beta}\circ\mu_{\alpha,\,\beta}\circ(T_\alpha\ot \id)+\lambda T_{1,\,\alpha\beta}\circ \mu_{\alpha,\,\beta}\right).
	\end{array}\]
	
	Note that  the first equation is exactly
	$\delta_{\rm{Alg}}(\mu_{1})_{\alpha,\,\beta,\,\gamma}=0$, then $\delta_{\rm{Alg}}(\mu_{1})=0\in \mathcal{C}^\bullet_{\Alg}(A)$,
	and the second equation is equivalent to
	$ \Phi(\mu_{1})_{\alpha,\,\beta}=-\partial_{\rm{RBO_\lambda}}(T_{1})_{\alpha,\,\beta} $,
	% \[\Phi^2(\mu_{1})_{\alpha,\,\beta}=-\partial_{\rm{RBO_\lambda}}^1(P_{1})_{\alpha,\,\beta}
	%% \in \mathcal{C}^\bullet_{\RBO}(A)
	%.\]
	then
	$ \Phi(\mu_{1})=-\partial_{\rm{RBO_\lambda}}(T_{1}) \in \mathcal{C}^\bullet_{\RBO}(T). $
	%{\color{blue}
		%	then
		%	\[\Phi^2(\mu_{1})=-\partial_{\rm{RBO_\lambda}}^1(P_{1}) \in \mathcal{C}^\bullet_{\RBO}(A).\]
		%}
	So $(\mu_1, T_1)$
%$(
%	%	\mu_{1,\,\alpha,\,\beta}
%	(\mu_{1,\,\alpha,\,\beta})_{\alpha,\,\beta \in \Omega}
%	,(P_{1,\,\alpha})_{\alpha\in\Omega})$
	is a 2-cocycle in $\mathcal{C}^\bullet_{\RBA}(A)$.
\end{proof}

\begin{defn}
	Let $(A[[t]], \mu_t, T_t)$ and $(A[[t]], \mu'_t, T'_t)$
	 be two $1$-parameter formal deformations of $\Omega$-family Rota-Baxter algebra $(A, \mu, T)$
of weight $\lambda$. {\bf A formal isomorphism} from $(A[[t]], \mu'_t, T'_t)$
%	$(A[[t]],$ $
%	%\mu_{t,\,\alpha,\,\beta}'
%	(\mu_{t,\,\alpha,\,\beta}')_{\alpha,\,\beta \in \Omega}, (T_{t,\,\omega}')_{\omega\in\Omega})$
	to $(A[[t]], \mu_t, T_t)$
%	$(A[[t]],
%	%\mu_{t,\,\alpha,\,\beta}
%	(\mu_{t,\,\alpha,\,\beta})_{\alpha,\,\beta \in \Omega}
%	,(T_{t,\,\omega})_{\omega\in\Omega})$
	is a family of power series $ \psi_{t} = \{\psi_{t,\theta}\}_{\theta \in \Omega}$, where $\psi_{t,\,\theta}=\sum_{i=0}\psi_{i,\,\theta}t^i: A[[t]]\rightarrow A[[t]]$
	%, where
	and $\psi_{i,\,\theta}: A\rightarrow A$ are linear maps with $\psi_{0,\,\theta}=\id_A$, such that:
	\begin{eqnarray}\label{Eq: equivalent deformations}
		\psi_{t,\,\alpha\beta}\circ \mu_{t,\,\alpha,\,\beta}' &=& \mu_{t,\,\alpha,\,\beta}\circ (\psi_{t,\,\alpha}\ot \psi_{t,\,\beta}),\\
		\psi_{t,\,\omega}\circ T_{t,\,\omega}'&=&T_{t,\,\omega}\circ\psi_{t,\,\omega}. \label{Eq: equivalent deformations2}
	\end{eqnarray}
	In this case, we say that the two 1-parameter formal deformations $(A[[t]], \mu_t, T_t)$
	and $(A[[t]], \mu'_t, T'_t)$
	are {\bf equivalent}.
\end{defn}

\smallskip

Given an $\Omega$-family Rota-Baxter algebra $(A, \mu, T)$  of weight $\lambda$
%$(A,
%%\mu_{\alpha,\,\beta}
%(\mu_{\alpha,\,\beta})_{\alpha,\,\beta \in \Omega}
%,(T_\omega)_{\omega\in\Omega})$
, the power series
%$
%%\mu_{t,\,\alpha,\,\beta}
%(\mu_{t,\,\alpha,\,\beta})_{\alpha,\,\beta \in \Omega}$
%with $\mu_{i,\,\alpha,\,\beta}=\delta^{i}_{0}\mu_{\alpha,\,\beta}$
$\mu_t$ with $\mu_{i}=\delta^{i}_{0} \mu$
and
%$T_{t,\,\omega}$ with $T_{i,\,\omega}=\delta^{i}_{0}T_\omega$
$T_t$ with $T_{i}=\delta^{i}_{0} T$
make
%$(A[[t]],
%%\mu_{t,\,\alpha,\,\beta}
%(\mu_{t,\,\alpha,\,\beta})_{\alpha,\,\beta \in \Omega}
%,(T_{t,\,\omega})_{\omega\in\Omega})$
$(A[[t]], \mu_t, T_t)$
into a $1$-parameter formal deformation of $(A, \mu, T)$,
%$(A,
%%\mu_{\alpha,\,\beta}
%(\mu_{\alpha,\,\beta})_{\alpha,\,\beta \in \Omega}
%,(T_\omega)_{\omega\in\Omega})$,
where $ \delta^{i}_{j} $
is the  Kronecker delta symbol. Formal deformations equivalent to this one are called {\bf trivial}.
\smallskip

\begin{thm}
	The infinitesimals of two equivalent 1-parameter formal deformations of $(A, \mu, T)$
	are in the same cohomology class in $\rmH^\bullet_{\RBA}(A,T)$.
\end{thm}

\begin{proof}
	Let $ \psi : (A[[t]], \mu'_t, T'_t) \rightarrow  (A[[t]], \mu_t, T_t)$	
	be a formal isomorphism.
	Expanding the identities and collecting coefficients of $t$, we get from Eqs.~(\ref{Eq: equivalent deformations}) and (\ref{Eq: equivalent deformations2})
	\begin{eqnarray*}
		\mu_{1,\,\alpha,\,\beta}'&=&\mu_{1,\,\alpha,\,\beta}+\mu_{\alpha,\,\beta}\circ(\id\ot \psi_{1,\,\beta})-\psi_{1,\,\alpha\beta}\circ\mu_{\alpha,\,\beta}+\mu_{\alpha,\,\beta}\circ(\psi_{1,\,\alpha}\ot \id),\\
		T'_{1,\,\omega}&=&T_{1,\,\omega}+T_\omega\circ\psi_{1,\,\omega}-\psi_{1,\,\omega}\circ T_\omega,
	\end{eqnarray*}
	that is, we have\[(\mu_{1,\,\alpha,\,\beta}',T'_{1,\,\omega})-(\mu_{1,\,\alpha,\,\beta},T_{1,\,\omega})
	=({\delta_{\rm{Alg}}(\psi_1)}_{\alpha,\,\beta}, -{\Phi(\psi_1)}_{\omega})=\partial_{\rm{RBA}_\lambda}(\psi_1,0)_{\alpha,\,\beta,\,\omega}
	%\in  \mathcal{C}^\bullet_{\RBA}(A)
	,\]
	hence $ (\mu_{1}',T'_{1})-(\mu_{1},T_{1})
	=\partial_{\rm{RBA}_\lambda}(\psi_1,0)\in  \mathcal{C}^\bullet_{\RBA}(A). $
%	\[(\mu_{1}',T'_{1})-(\mu_{1},T_{1})
%	=({\delta_{\rm{Alg}}(\psi_1)}, -{\Phi(\psi_1)})=\partial_{\rm{RBA}_\lambda}(\psi_1,0)\in  \mathcal{C}^\bullet_{\RBA}(A).\]
\end{proof}

\begin{defn}
	An $\Omega$-family Rota-Baxter algebra $ (A, \mu, T) $ of weight $\lambda$
%	$(A,
%	%	\mu_{\alpha,\,\beta}
%	(\mu_{\alpha,\,\beta})_{\alpha,\,\beta \in \Omega},(T_\omega)_{\omega\in\Omega})$
	 is said to be {\bf rigid} if its arbitrary 1-parameter formal deformation is trivial.
\end{defn}

\begin{thm}
	Let  $ (A, \mu, T) $ be an  $\Omega$-family Rota-Baxter algebra of weight $\lambda$. If  $\rmH^2_{\RBA}(A)=0$, then $ (A, \mu, T) $ is rigid.
\end{thm}

\begin{proof}Let $(A[[t]], \mu_t, T_t)$
%	$(A[[t]],
%	(\mu_{t,\,\alpha,\,\beta})_{\alpha,\,\beta \in \Omega},(T_{t,\,\omega})_{\omega\in\Omega})$
	 be a $1$-parameter formal deformation.
	%of $(A,
	%%	  \mu_{\alpha,\,\beta}
	%	(\mu_{\alpha,\,\beta})_{\alpha,\,\beta \in \Omega}
	%	  ,(T_\omega)_{\omega\in\Omega})$.
	By Proposition~\ref{Prop: Infinitesimal is 2-cocyle},
	$(\mu_1, T_1)$
%	$(
%	%\mu_{1,\,\alpha,\,\beta}
%	(\mu_{1,\,\alpha,\,\beta})_{\alpha,\,\beta \in \Omega}
%	, (T_{1,\,\alpha})_{\alpha\in\Omega})$
	is a $2$-cocycle. By $\rmH^2_{\RBA}(A,T)=0$, there exists a $1$-cochain
	%	$$(\psi'_{1}, x) \in \mathcal{C}^1_\RBA(A,T)= \mathcal{C}^1_{\Alg}(A)\oplus \Hom(\bfk, A)$$
	$(\psi_{1}, 0) \in \mathcal{C}^1_\RBA(A,T)= \mathcal{C}^1_{\Alg}(A)$
	such that $ (\mu_{1}, T_{1}) =  \partial_{{\rm{RBA}_\lambda}}(\psi_1, 0) $,
%		\[(\mu_{1}, T_{1}) =  \partial_{{\rm{RBA}_\lambda}}^1(\psi_1, 0), \]
	%\[(\mu_{1,\,\alpha,\,\beta}, T_{1,\,\omega}) =  \partial_{{\rm{RBA}_\lambda}}^1(\psi'_1, x)_{\alpha,\,\beta,\,\omega}, \]
%	\[(\mu_{1,\,\alpha,\,\beta}, T_{1,\,\omega}) =  \partial_{{\rm{RBA}_\lambda}}^1(\psi_1, 0)_{\alpha,\,\beta,\,\omega}, \]
	that is,
	%\[\mu_{1,\,\alpha,\,\beta}=\delta_{\rm{Alg}}^1(\psi_1')_{\alpha,\,\beta},\, \text{ and }\, T_{1,\,\omega}=-{\partial_{\rm{RBO_\lambda}}^0(x)}_\omega-{\Phi^1(\psi_1')}_\omega.\]
	\[\mu_{1,\,\alpha,\,\beta}=\delta_{\rm{Alg}}(\psi_1)_{\alpha,\,\beta},\, \text{ and }\, T_{1,\,\omega}=-{\Phi(\psi_1)}_\omega.\]
	% Let \[\psi_{1,\,\omega}=\psi'_{1,\,\omega}+\delta_{\rm{Alg}}^0(x)_\omega.\]
	%  Then we have
	% \[\mu_{1,\,\alpha,\,\beta}= \delta_{\rm{Alg}}^1(\psi_1)_{\alpha,\,\beta},\,\text{ and }\, T_{1,\,\omega}=-\Phi^1(\psi_1)_\omega,\]
	% as it can be readily seen that \[\Phi^1({\delta_{\rm{Alg}}^0})_\omega(x)={\partial_{\RBO,\,\omega}^0}(x).\]
	
	Setting $\psi_{t} = \Id_A -\psi_{1}t$
%	$\psi_{t,\,\alpha} = \Id_A -\psi_{1,\,\alpha}t$
	, we have a deformation
	$(A[[t]], \overline \mu_{t}, \lbar{T}_{t})$
%	$(A[[t]],
%	% \overline{\mu}_{t,\alpha,\,\beta}
%	(\overline \mu_{t,\,\alpha,\,\beta})_{\alpha,\,\beta \in \Omega}
%	, (\lbar{T}_{t, \omega})_{\omega\in\Omega})$
	, where
	$$\overline{\mu}_{t,\,\alpha,\,\beta}=\psi^{-1}_{t,\,\alpha\beta}\circ \mu_{t,\,\alpha,\,\beta}\circ (\psi_{t,\,\alpha}\otimes \psi_{t,\,\beta})$$
	and $$\lbar{T}_{t,\,\omega}=\psi^{-1}_{t,\,\omega}\circ T_{t,\,\omega}\circ \psi_{t,\,\omega}.$$
	It can be easily verify  that $\overline{\mu}_{1,\,\alpha,\,\beta}=0$ and $\overline{T}_{1,\,\omega}=0$. Then
	$$\begin{array}{rcl} \overline{\mu}_{t,\,\alpha,\,\beta}&=& \mu_{\alpha,\,\beta}+\overline{\mu}_{2,\,\alpha,\,\beta}t^2+\dots,\\
		\lbar{T}_{t,\,\omega}&=& T_\omega+\overline{T}_{2,\,\omega}t^2+\dots.\end{array}$$
	By Eqs.~(\ref{Eq: deform eq for  products in RBA}) and (\ref{Eq: Deform RB operator in RBA}),
	We see that
	% $(\overline{\mu}_{2,\,\alpha,\,\beta},  \overline{P}_{2,\,\omega})$
%	$((\overline{\mu}_{2,\,\alpha,\,\beta})_{\alpha,\,\beta \in \Omega} ,  (\overline{T}_{2,\,\omega})_{\omega \in \Omega} )$
	$ (\overline \mu_{2}, \lbar{T}_{2}) $
	is still a $2$-cocyle. By induction, we can show that $(A[[t]], \mu_t, T_t)$
	%   $ (A[[t]], \mu_{t,\,\alpha,\,\beta} , P_{t,\omega}) $
%	$ (A[[t]], (\mu_{t,\,\alpha,\,\beta})_{\alpha,\,\beta \in \Omega} , (T_{t,\,\omega})_{\omega \in \Omega}) $
	is equivalent to the trivial extension $(A[[t]], \mu, T)$.
%	$(A[[t]], (\mu_{\alpha,\,\beta})_{\alpha,\,\beta \in \Omega},$ $ (T_\omega)_{\omega\in\Omega}).$
	Thus, $(A, \mu, T)$
% $(A,(\mu_{\alpha,\,\beta})_{\alpha,\,\beta\in\Omega},(T_\omega)_{\omega\in\Omega})$
	is rigid.
\end{proof}

\smallskip
\section{Appendix: The formulas of the coboundary operator on $\mathcal{C}_{\rm RelRBA_\lambda}(A,V)$}\

In this appendix, we will describe the coboundary operator on the cochain complex of relative $\Omega$-family Rota-Baxter algebras of weight $\lambda$ clearly. Explicitly, we will give precise formulas for the operators $\delta_\pi$, $d_T$ and $h_T$ on each component of $\mathcal{C}_{\rm RelRBA_\lambda}(A,V)$, as stated in Remark~\ref{remark appendix}.
%Now we give the coboundary operators $\delta_\pi, d_T$ and the cochain map $h_T$ in Subsection~\ref{sub:cohomology}

 Firstly, let's introduce some notations.
\begin{defn} Let $[n]$ be the ordered set $\{1,2,\dots,n\}$, $n\geqslant 1$.
	Define
	\begin{align*}
		P^{n}_{0} := & \left\{ \emptyset \right\}, \\
		P^{n}_{i} := & \Big\{ \text{all ordered subsets of }  [n] \mbox{\ of cardinality $i$}\Big\}, 1 \leqslant i \leqslant n,\\
		P([n]) := & \bigsqcup_{i=0}^{n} P_{i}^{n}.
	\end{align*}
\end{defn}

\begin{remark}\label{order}
	For any $ n \geqslant 1 $,the set  $ P([n]) $ is totally ordered as follows: for any $ 1 \leqslant i,j \leqslant n $, $ I = \{a_{1}, \dots, a_{i}\} \in P^{n}_{i} , J= \{b_{1}, \dots, b_{j}\} \in P^{n}_{j}$, we set $I\prec J$ if $i<j$, or $i=j$ and $I$ is smaller than $J$ with respect to the lexicographical order.
	With this order, we have
	\[ \emptyset \prec \{1\} \prec \dots \prec \{n\} \prec \{1,2\} \prec \{1,3\} \prec \dots \prec \{1,n\} \prec \{2,3\} \prec \dots \prec \{1,\dots,n\}. \]
\end{remark}

\begin{defn}
	Let $ A, V $ be two vector spaces.
	%	 $ n \geqslant 1 $,
	%	 $ 0 \leqslant j \leqslant n $, define $ \mathcal{A}^{j,n-j} $ to be the subspace of $(A\oplus V)^{\otimes n}$ consisting of the tensor powers of $A$ and $V$ with $A$, $V$ appearing $j,n-j$ times respectively. For $ 1 \leqslant i \leqslant n $, take $ I = \{q_{1},\dots,q_{i}\} \in P^{n}_{i} $,
	Let  $ I = \{q_{1},\dots,q_{i}\} \in P^{n}_{i} $,  $ 1 \leqslant i \leqslant n $.
	Define a subspace
	%	  $ \mathcal{A}^{n}_{I} $
	of
	%	  $(A\oplus V)^{\otimes n}$
	$ \mathcal{A}^{i,n-i} $
	as
	\[ \mathcal{A}^{n}_{I} := V  \ot \dots \ot V \ot \underset{q_{1}}{A} \ot V \ot \dots \ot V \ot \underset{q_{i}}{A} \ot V \ot\dots \ot V. \]
	In particular, define $ \mathcal{A}^{n}_{\emptyset} := V^{\ot n} $.
\end{defn}

With respect to the order $\prec$ introduced above, the space $\Hom ((A \oplus V)^{\ot n}, A \oplus V)$ can be decomposed as
\begin{align*}
	&	\Hom ((A \oplus V)^{\ot n}, A \oplus V)\\
	= \ & \Hom ((A \oplus V)^{\ot n}, A) \oplus \Hom ((A \oplus V)^{\ot n},V) \\
	= \ & \bigoplus_{i=0}^{n} \Hom (\mathcal{A}^{i,n-i}, A) \bigoplus
	\bigoplus_{j=0}^{n} \Hom (\mathcal{A}^{j,n-j}, V) \\
	= \ & \bigoplus_{i=0}^{n} \bigoplus_{I \in P^{n}_{i}} \Hom 		(\mathcal{A}^{n}_{I}, A) \bigoplus
	\bigoplus_{j=0}^{n} \bigoplus_{I \in P^{n}_{j}} \Hom 		(\mathcal{A}^{n}_{I}, V),		
\end{align*}
then any element $f\in \Hom ((A \oplus V)^{\ot n}, A \oplus V)$ can be written as:
\begin{equation}\label{map decomposition}
	\begin{aligned}
		f
%		= \ & f^{A}_{\emptyset}+f^{A}_{\{1\}}+\dots+f^{A}_{\{1,\dots,n\}}+f^{V}_{\emptyset}+f^{V}_{\{1\}}+\dots+f^{V}_{\{1,\dots,n\}} \\
		= \ & (f^{A}_{\emptyset},f^{A}_{\{1\}},\dots,f^{A}_{\{1,\dots,n\}};f^{V}_{\emptyset},f^{V}_{\{1\}},\dots,f^{V}_{\{1,\dots,n\}}),
	\end{aligned}
\end{equation}
with $ f_{I}^{A} \in \Hom(\mathcal{A}^{n}_{I},A) $ and $ f_{I}^{V} \in \Hom(\mathcal{A}^{n}_{I},V) $, $ I \in P([n]) $. Thereafter, we will use this notation flexibly without additional explanation.

\begin{defn}
	Let $I=\{p_1,\ldots,p_i\}\in P([n]), J=\{q_1,\ldots,q_j\}\in P([m]),k\in\{1,\ldots,n\}$. Define
	\begin{align*}
		I \sideset{_k}{}{\mathop{\sqcup}} J=& \ \{p_1,\ldots, p_{t-1}, k+q_1-1,\ldots,k+q_j-1, p_{t+1}+j-1,\ldots, p_{i}+j-1\}, k=p_t, 1\leqslant t\leqslant i,\\
		%I \sideset{_k}{}{\mathop{\sqcup}} J=& \ \{p_1,\ldots, p_{t-1}, p_t+q_1-1,\ldots,p_t+q_j-1, p_{t+1}+j-1,\ldots, p_{i}+j-1\}, k=p_t, 1\leqslant t\leqslant i,\\
		I\sqcup_k J=& \
		\left\{
		\begin{array}{rcl}
			\{k+q_1-1,\ldots,k+q_{j}-1, p_1+j-1,\ldots,p_i+j-1\}, & 1\leqslant k<p_1,\\
			\{p_1,\ldots,p_t, k+q_1-1,\ldots,k+q_j-1, p_{t+1}+j-1,\ldots,p_i+j-1\}, & p_t<k<p_{t+1},  1\leqslant t\leqslant i-1,\\
			\{p_1,\ldots,p_i,k+q_1-1,\ldots,k+q_{j}-1\}, & p_i<k\leqslant n.
		\end{array}
		\right.
	\end{align*}
\end{defn}
For any $f\in\Hom((A\oplus V)^{\ot n}, A\oplus V), g\in \Hom((A\oplus V)^{\ot m}, A\oplus V)$, we have
\[[f,g]_{\Omega}\in\Hom((A\oplus V)^{\ot n+m-1}, A\oplus V).\]
For $I\in P([n+m-1])$, the projections of $[f,g]_\Omega$ on the components $\Hom(\mathcal{A}^n_I,A)$ and $\Hom(\mathcal{A}^n_I,V)$  are given as:
\begin{align*}
	{[f,g]_{\Omega}}^A_{I}=& \ \sum_{k=1}^n\sum_{I_1\sideset{_k}{}{\mathop{\sqcup}}I_2=I}(-1)^{(k-1)(m-1)}f^A_{I_1}\circ_k g^A_{I_2}
	+\sum_{k=1}^n\sum_{I_1\sqcup_k I_2=I}(-1)^{(k-1)(m-1)}f^A_{I_1}\circ_k g^V_{I_2}\\
	&-(-1)^{(n-1)(m-1)}\sum_{k=1}^m\sum_{I_1\sideset{_k}{}{\mathop{\sqcup}} I_2=I}(-1)^{(k-1)(n-1)}g^A_{I_1}\circ_k f^A_{I_2}\\
	& \ -(-1)^{(n-1)(m-1)}\sum_{k=1}^m\sum_{I_1\sqcup_k I_2=I}(-1)^{(k-1)(n-1)}g^A_{I_1}\circ_k f^V_{I_2},\\
	{[f,g]_{\Omega}}^V_{I}=& \ \sum_{k=1}^n\sum_{I_1\sideset{_k}{}{\mathop{\sqcup}}I_2=I}(-1)^{(k-1)(m-1)}f^V_{I_1}\circ_k g^A_{I_2}
	+\sum_{k=1}^n\sum_{I_1\sqcup_k I_2=I}(-1)^{(k-1)(m-1)}f^V_{I_1}\circ_k g^V_{I_2}\\
	&-(-1)^{(n-1)(m-1)}\sum_{k=1}^m\sum_{I_1\sideset{_k}{}{\mathop{\sqcup}} I_2=I}(-1)^{(k-1)(n-1)}g^V_{I_1}\circ_k f^A_{I_2}\\
	& \ -(-1)^{(n-1)(m-1)}\sum_{k=1}^m\sum_{I_1\sqcup_k I_2=I}(-1)^{(k-1)(n-1)}g^V_{I_1}\circ_k f^V_{I_2}.
\end{align*}
%$ \oplus_{n=0}^{+\infty} \left(\mathrm{Hom}_\Omega(A^{\ot n+1},A) \oplus \oplus_{i=0}^{n} \mathrm{Hom}_\Omega(\mathcal{A}^{i,n+1-i},V)\right) $

In particular, the restriction $[-,-]_\Omega$ on $ \frak{L}' $ is given as: for any
\begin{align*}
	f = \ &(0,\ldots,0,f^A_{\{1,\ldots,n\}}; f^V_\emptyset, f^V_{\{1\}},\ldots,f^V_{\{2,\ldots,n\}}, 0) \in\frak{L}'(n),\\
	g = \ &(0,\ldots,0,g^A_{\{1,\ldots,m\}}; g^V_\emptyset, g^V_{\{1\}},\ldots,g^V_{\{2,\ldots,n\}}, 0) \in\frak{L}'(m),
\end{align*}
%for $
%  \forall f\in\frak{L}'(n)=\hm(A^{\ot n}, A)\oplus\oplus_{i=0}^{n-1}\hm(\cal{A}^{i, n-i}, V),
%$ we have
%\[f=(0,\ldots,0,f^A_{\{1,\ldots,n\}}; f^V_\emptyset, f^V_{\{1\}},\ldots,f^V_{\{2,\ldots,n\}}, 0) \in\frak{L}'(n),\]
%For $
%  \forall g\in\frak{L}'(m)=\hm(A^{\ot m}, A)\oplus\oplus_{i=0}^{m-1}\hm(\cal{A}^{i, m-i}, V),
%$ we have
%\[g=(0,\ldots,0,g^A_{\{1,\ldots,m\}}; g^V_\emptyset, g^V_{\{1\}},\ldots,g^V_{\{2,\ldots,n\}}, 0) \in\frak{L}'(m) .\]
we have
\begin{align} \label{special f,g,G}
	[f,g]_\Omega
	= \ \Big(0,\ldots,0,{[f,g]_{\Omega}}^A_{\{1,\ldots,n+m-1\}}, {[f,g]_{\Omega}}^V_{\emptyset}, {[f,g]_{\Omega}}^V_{\{1\}},\ldots,
	{[f,g]_{\Omega}}^V_{\{2,\ldots,n+m-1\}},0\Big) \in\frak{L}'(n+m-1) ,
\end{align}
%\begin{align*}
%&[f,g]_G\\
%=& \ \Big(0,\ldots,0,{[f,g]_{\Omega}}^A_{\{1,\ldots,n+m-1\}}, {[f,g]_{\Omega}}^V_{\emptyset}, {[f,g]_{\Omega}}^V_{\{1\}},\ldots,
%{[f,g]_{\Omega}}^V_{\{2,\ldots,n+m-1\}},0\Big) \in\frak{L}'(n+m-1) ,
%\end{align*}
where
\begin{align*}
	{[f,g]_{\Omega}}^A_{\{1,\ldots,n+m-1\}}
	=& \ \sum_{i=1}^n(-1)^{(i-1)(m-1)}f^A_{\{1,\ldots,n\}}\circ_ig^A_{\{1,\ldots,m\}}
	-(-1)^{(n-1)(m-1)}\sum_{i=1}^m(-1)^{(i-1)(n-1)}g^A_{\{1,\ldots,m\}}\circ_if^A_{\{1,\ldots,n\}}\\
	=& \ [f^A_{\{1,\ldots,n\}}, g^A_{\{1,\ldots,m\}}]_{\Omega},\\
	{[f,g]_{\Omega}}^V_{\emptyset}
	=& \ \sum_{i=1}^n(-1)^{(i-1)(m-1)}f^V_\emptyset\circ_ig^V_\emptyset
	-(-1)^{(n-1)(m-1)}\sum_{i=1}^m(-1)^{(i-1)(n-1)}g^V_{\emptyset}\circ_if^V_{\emptyset}\\
	=& \ [f^V_\emptyset,g^V_\emptyset]_{\Omega},
\end{align*}
and
\begin{align*}
	{[f,g]_{\Omega}}^V_{I}=& \ \sum_{k=1}^n\sum_{\substack{I_2=\{1,\ldots,m\}\\   I_1\sideset{_k}{}{\mathop{\sqcup}}I_2=I}}              (-1)^{(k-1)(m-1)}f^V_{I_1}\circ_kg^A_{I_2}\\
	&+\sum_{k=1}^n\sum_{\substack{I_2\neq\{1,\ldots,m\}\\   I_1\sqcup_k I_2=I}}              (-1)^{(k-1)(m-1)}f^V_{I_1}\circ_kg^V_{I_2}\\
	&-(-1)^{(n-1)(m-1)}\sum_{k=1}^m\sum_{\substack{I_2=\{1,\ldots,n\}\\   I_1\sideset{_k}{}{\mathop{\sqcup}}I_2=I}}             (-1)^{(k-1)(n-1)}g^V_{I_1}\circ_kf^A_{I_2}\\
	&-(-1)^{(n-1)(m-1)}\sum_{k=1}^m\sum_{\substack{I_2\neq\{1,\ldots,n\}\\   I_1\sqcup_kI_2=I}}             (-1)^{(k-1)(n-1)}g^V_{I_1}\circ_kf^V_{I_2}.
\end{align*}
for $I \in P([n])\backslash\{\emptyset,\{1,\ldots,n\}\}$.
With above notations, we can give a clear description of the coboundary operators $\delta_\pi, d_T$ and the cochain map $h_T$ clearly on each component of $\mathcal{C}^\bullet_{\rm RelRBA_\lambda}(A,V)$.
We write $$ \pi = (\pi_{\emptyset}^{A},\pi_{\{1\}}^{A},\pi_{\{2\}}^{A},\pi_{\{1,2\}}^{A};\pi_{\emptyset}^{V},\pi_{\{1\}}^{V},\pi_{\{2\}}^{V},\pi_{\{1,2\}}^{V}) = (0,0,0,\mu;\mu_{V},l,r,0)  \in \mathcal{C}_{\rm AssAct}^2(A, V).$$
For any $ f = (0,\ldots,0,f^A_{\{1,\ldots,n\}}; f^V_\emptyset, f^V_{\{1\}},\ldots,f^V_{\{2,\ldots,n\}}, 0)  \in \mathcal{C}_{\rm AssAct}^n(A,V) $, $ \theta \in  \mathcal{C}_{\rm RelRBO_\lambda}^{n-1}(T)$.
%For $
%  f\in\frakC^n(A, l, r)=\hm(A^{\ot n}, A)\oplus\oplus_{i=0}^{n-1}\hm(\cal{A}^{i, n-i}, V),
%$
By Eq.~\eqref{special f,g,G}, we have
%\begin{align*}
%f=& \ (0,\ldots,0,f^A_{\{1,\ldots,n\}}; f^V_\emptyset, f^V_{\{1\}},\ldots,f^V_{\{2,\ldots,n\}}, 0)\\
%=& \ f^A_{\{1,\ldots,n\}}+f^V_\emptyset+f^V_{\{1\}}+\ldots+f^V_{\{2,\ldots,n\}},
%\end{align*}
%then
\begin{equation}\label{relative delta}
	\begin{aligned}
		\delta_\pi(f)=& \ [\pi, f]_{\Omega} \\
		=& \ \Big(0,\ldots,0,\delta_\pi(f)^A_{\{1,\ldots,n+1\}}; \delta_\pi(f)^V_\emptyset,\ldots,\delta_\pi(f)^V_{\{2,\ldots,n+1\}}, 0\Big) \in \mathcal{C}_{\rm AssAct}^{n+1}(A,V) ,
	\end{aligned}
\end{equation}
where
\begin{align*}
	\delta_\pi(f)^A_{\{1,\ldots,n+1\}}
	=& \ \mu\circ(f^A_{\{1,\ldots,n\}}\ot\id)+(-1)^{n-1}\mu\circ(\id\ot f^A_{\{1,\ldots,n\}})
	+\sum_{i=1}^n(-1)^{n-i+1}f^A_{\{1,\ldots,n\}}\circ_i \mu \\
	=& \ [\mu,f^A_{\{1,\ldots,n\}}]_{\Omega},\\
	\delta_\pi(f)^V_\emptyset
	=& \ \mu_V\circ(f^V_\emptyset\ot\id)+(-1)^{n-1}\mu_V\circ(\id\ot f^V_\emptyset)
	+\sum_{i=1}^n(-1)^{n-i+1}f^V_\emptyset\circ_i\mu_V \\
	=& \ [\mu_V, f^V_\emptyset]_{\Omega};
\end{align*}

\begin{align*}
	\delta_\pi(f)^V_{\{q\}}=& \ \delta^q_1(-1)^{n-1}l\circ(\id\ot f^V_\emptyset)+(-1)^{n-q+1}f^V_\emptyset\circ_q l
	+\delta^q_{n+1}r\circ(f^V_\emptyset\ot\id)+(-1)^{n-(q-1)+1}f^V_\emptyset\circ_{q-1}r\\
	&+(1-\delta^q_{n+1})\mu_V\circ(f^V_{\{q\}}\ot\id)+(1-\delta^q_1)(-1)^{n-1}\mu_V\circ(\id\ot f^V_{\{q-1\}})
	+\sum_{i=1}^{q-2}(-1)^{n-i+1}f^V_{\{q-1\}}\circ_i\mu_V\\
	& \ +\sum_{i=q+1}^n(-1)^{n-i+1}f^V_{\{q\}}\circ_i\mu_V
\end{align*}
for $\{q\} \in P^{n+1}_{1}$;
\begin{align*}
	\delta_\pi(f)^V_{\{q_1,\ldots,q_k\}}=& \ \sum_{j=1}^{k-1}\delta^{q_{j+1}}_{q_j+1}(-1)^{n-q_i+1}f^V_{\{q_1,\ldots,q_j,q_{j+2}-1,\ldots, q_k-1\}}\circ_{q_j}\mu
	+\delta^{q_1}_1(-1)^{n-1}l\circ(\id\ot f^V_{\{q_2-1,\ldots,q_k-1\}})\\
	&+\sum_{j=1}^{k}(1-\delta^{q_{j+1}}_{q_j+1})(-1)^{n-q_j+1}f^V_{\{q_1,\ldots,q_{j-1},
		q_{j+1}-1,\ldots, q_k-1\}}\circ_{q_j}l
	+\delta^{q_k}_{n+1}r\circ(f^V_{\{q_1,\ldots,q_{k-1}\}}\ot\id)\\
	&+\sum_{j=1}^{k}(1-\delta^{q_j-1}_{q_{j-1}})(-1)^{n-(q_j-1)+1}f^V_{\{q_1,\ldots,q_{j-1},
		q_{j+1}-1,\ldots, q_k-1\}}\circ_{q_j -1}r
	+(1-\delta^{q_k}_{n+1})\mu_V\circ(f^V_{\{q_1,\ldots,q_k\}}\ot\id)\\
	&+(1-\delta^{q_1}_1)(-1)^{n-1}\mu_V\circ(\id\ot f^V_{\{q_1-1,\ldots,q_k-1\}})
	+\sum_{i=1}^{q_1-2}(-1)^{n-i+1}f^V_{\{q_1-1,\ldots,q_k-1\}}\circ_i\mu_V\\
	&+\sum_{j=1}^{k-1}\sum_{i=q_j+1}^{q_{j+1}-2}(-1)^{n-i+1}f^V_{\{q_1,\ldots,q_j,q_{j+1}-1,\ldots, q_k-1\}}\circ_i\mu_V
	+\sum_{i=q_k+1}(-1)^{n-i+1}f^V_{\{q_1,\ldots, q_k\}}\circ_i\mu_V
\end{align*}
for $2\leqslant k\leqslant n-1, {\{q_1,\ldots,q_k\}} \in P^{n+1}_{k}$;
and 
\begin{align*}
	\delta_\pi(f)^V_{\{1,\ldots,\hat{q},\ldots,n+1\}}
	=& \ \sum_{i=1}^{q-2}(-1)^{n-i+1}f^V_{\{1,\ldots,\widehat{q-1},\ldots,n\}}\circ_{i}\mu
	+\sum_{i=q+1}^n(-1)^{n-i+1}f^V_{\{1,\ldots,\widehat{q},\ldots,n\}}\circ_{i}\mu\\
	&+(1-\delta^q_1)(-1)^{n-1}l\circ(\id\ot f^V_{\{1,\ldots,\widehat{q-1},\ldots,n\}})
	+(-1)^{n-{(q-1)}+1} f^V_{\{1,\ldots,\widehat{q-1},\ldots,n\}}\circ_{q-1}l\\
	&+(1-\delta^q_{n+1})r\circ(f^V_{\{1,\ldots,\widehat{q},\ldots,n\}}\ot\id)
	+(-1)^{n-q+1} f^V_{\{1,\ldots,\widehat{q},\ldots,n\}}\circ_{q}r\\
	&+\delta^q_{n+1}l\circ(f^A_{\{1,\ldots,n\}}\ot\id)
	+\delta^q_1(-1)^{n-1}r\circ(\id\ot f^A_{\{1,\ldots,n\}}).
\end{align*}
and for $ 1 \leqslant q \leqslant n+1, \{1,\ldots,\hat{q},\ldots,n+1\} \in P^{n+1}_{n} $,
where $\delta^i_j$ is the Kronecker symbol.
Also, we have

\begin{equation}\label{relative partial}
	\begin{aligned}
		d_T(\theta)=& \ \lambda[\mu_V, \theta]_{\Omega}+\big[[\pi,T]_{\Omega}, \theta\big]_{\Omega}\\
		=& \ (-1)^{n}\left(\mu\circ(T\ot\theta)-T\circ r\circ(\id\ot\theta)\right)+\sum_{i=1}^{n-1}(-1)^{n-i}\theta
		\circ_{i} (l\circ(T\ot\id)+r\circ(\id\ot T)+\lambda\mu_V)\\
		&+\mu\circ(\theta\ot T)-T\circ l \circ(\theta\ot\id),
	\end{aligned}
\end{equation}
\begin{equation}\label{relative chain map}
	\begin{aligned}
		h_T(f)
		=& \ \sum_{k=1}^n\frac{1}{k!}\lambda^{n-k}P\big[\ldots \underbrace{[f,T]_{\Omega},\ldots, T}_k\big]_{\Omega}\\
		=& \ f^{A}_{\{1,\dots,n\}}\circ(\underbrace{T\ot\ldots\ot T}_n)\\
		& \ -\sum_{i=0}^{n-1}
		\sum_{I = \{q_{1},\dots,q_{i}\} \in P([n])}\lambda^{n-i-1}T\circ
		f^{V}_{I}
		\circ
		(\id \ot \dots \ot \id \ot \underset{q_{1}} {T} \ot \id \ot \dots \ot \id  \ot \underset{q_{i}} {T} \ot \id \ot \dots \ot \id).
	\end{aligned}
\end{equation}

\smallskip
\noindent
{{\bf Acknowledgments.} This work is supported in part by Natural Science Foundation of China (Grant No. 12101183,  12071137, 11971460) and  Science and Technology Commission of Shanghai Municipality (No. 22DZ2229014). The project is funded by China
Postdoctoral Science Foundation (Grant No. 2021M690049).

\end{document}